\documentclass[a4paper,twoside,11pt, reqno]{amsart} 
\usepackage{filecontents}
\newcommand{\Ueberschrift} {The \'etale topos reconstructs varieties over sub-$p$-adic fields}
\newcommand{\Kurztitel}{The \'etale topos reconstructs varieties over sub-$p$-adic fields}
\usepackage[margin=1.5in]{geometry}
\usepackage[headings]{fullpage}
\DeclareRobustCommand{\SkipTocEntry}[5]{}
\usepackage{thmtools}
\usepackage{amsmath}
\usepackage{amsthm}
\usepackage{mathtools}
\usepackage{mathrsfs}
\usepackage{amssymb}
\usepackage{nicefrac}
\usepackage{xcolor}
\definecolor{love}{RGB}{128, 15, 37}
\definecolor{slate}{RGB}{42, 54, 59}
\usepackage{stix2}
\usepackage[LGR,T1]{fontenc} 
\usepackage[utf8]{inputenc} 
\usepackage{url} 
\usepackage{hyperref}
\definecolor{darkblue}{rgb}{0,0,.85} 
\definecolor{darkred}{rgb}{0.84,0,0}
\hypersetup{
  colorlinks = true, 
  linktoc = all,     
  linkcolor = love,  
  citecolor = love,
  urlcolor=darkblue,       
  pdfencoding = unicode,
  hypertexnames = true,
  bookmarksdepth = subsubsection,
}

\usepackage[alphabetic,nobysame]{amsrefs}
\usepackage{csquotes}
\usepackage[noabbrev,nameinlink,capitalize]{cleveref} 
\newcommand{\stackstag}[1]{\href{https://stacks.math.columbia.edu/tag/#1}{Tag #1}}
\newcommand{\stacks}[1]{\cite[\stackstag{#1}]{stacks-project}}
\numberwithin{equation}{section} 

\theoremstyle{definition}

\newtheorem{definition}[equation]{Definition}

\newtheorem{recollection}[equation]{Recollection}
\newtheorem{remark}[equation]{Remark}
\newtheorem{remarks}[equation]{Remarks}
\newtheorem{recollections}[equation]{Recollections}

\newtheorem*{slogan*}{Slogan}

\theoremstyle{plain}

\newtheorem{corollary}[equation]{Corollary}
\newtheorem{lemma}[equation]{Lemma}
\newtheorem{proposition}[equation]{Proposition}
\newtheorem{theorem}[equation]{Theorem}

\newtheorem{thmABC}{Theorem}

\newtheorem{corABC}[thmABC]{Corollary}
\newtheorem{propABC}[thmABC]{Proposition}

\crefname{section}{\S\!}{\S\S\!}
\Crefname{section}{Section}{Sections}
\crefname{subsection}{\S\!}{\S\S\!}
\Crefname{subsection}{Subsection}{Subsections}

\crefformat{nul}{(#2#1#3)} 
\crefname{nul}{}{}
\Crefformat{nul}{(#2#1#3)}
\Crefname{nul}{}{}

\crefname{axiom}{Axiom}{Axioms}
\Crefname{axiom}{Axiom}{Axioms}
\crefname{convention}{Convention}{Conventions}
\Crefname{convention}{Convention}{Conventions}
\crefname{conjecture}{Conjecture}{Conjectures}
\Crefname{conjecture}{Conjecture}{Conjectures}
\crefname{construction}{Construction}{Constructions}
\Crefname{construction}{Construction}{Constructions}
\crefname{corollary}{Corollary}{Corollaries}
\Crefname{corollary}{Corollary}{Corollaries}
\crefname{counterexample}{Counterexample}{Counterexamples}
\Crefname{counterexample}{Counterexample}{Counterexamples}
\crefname{definition}{Definition}{Definitions}
\Crefname{definition}{Definition}{Definitions}
\crefname{example}{Example}{Examples}
\Crefname{example}{Example}{Examples}
\crefname{exercise}{Exercise}{Exercises}
\Crefname{exercise}{Exercise}{Exercises}
\crefname{lemma}{Lemma}{Lemmas} 
\Crefname{lemma}{Lemma}{Lemmas} 
\crefname{notation}{Notation}{Notations}
\Crefname{notation}{Notation}{Notations}
\crefname{observation}{Observation}{Observations}
\Crefname{observation}{Observation}{Observations}
\crefname{proposition}{Proposition}{Propositions}
\Crefname{proposition}{Proposition}{Propositions}
\crefname{question}{Question}{Questions}
\Crefname{question}{Question}{Questions}
\crefname{recollection}{Recollection}{Recollections}
\Crefname{recollection}{Recollection}{Recollections}
\crefname{remark}{Remark}{Remarks}
\Crefname{remark}{Remark}{Remarks}
\crefname{remarks}{Remarks}{Remarks}
\Crefname{remarks}{Remarks}{Remarks}
\crefname{subexample}{Subexample}{Subexamples}
\Crefname{subexample}{Subexample}{Subexamples}
\crefname{theorem}{Theorem}{Theorems}
\Crefname{theorem}{Theorem}{Theorems}
\crefname{variant}{Variant}{Variants}
\Crefname{variant}{Variant}{Variants}
\crefname{warning}{Warning}{Warnings}
\Crefname{warning}{Warning}{Warnings}
\usepackage[shortlabels]{enumitem}
\setlist[enumerate]{
     itemsep  = 0.2cm,
       label  = {\upshape (\arabic*)},
         ref  = \arabic*,
    leftmargin  = *
}

\newcommand{\enumref}[2]{{\upshape (\hyperref[#1.#2]{\ref*{#1}.\ref*{#1.#2}})}}
\usepackage{xspace}
\usepackage[final,factor=1100,stretch=10,shrink=10]{microtype}
\SetProtrusion{encoding={*},family={bch},series={*},size={6,7}}
              {1={ ,750},2={ ,500},3={ ,500},4={ ,500},5={ ,500},
               6={ ,500},7={ ,600},8={ ,500},9={ ,500},0={ ,500}}

\SetExtraKerning[unit=space]
    {encoding={*}, family={bch}, series={*}, size={footnotesize,small,normalsize}}
    {\textendash={400,400},   
     "28={ ,150},     
     "29={150, },     
     \textquotedblleft={ ,150},   
     \textquotedblright={150, }}  

\SetTracking{encoding={*}, shape=sc}{40}
\usepackage{tikz}
\usepackage{tikz-cd}
\usepackage[all]{xy} \CompileMatrices
\usepackage{extarrows}
\tikzcdset{arrow style=tikz, diagrams={>=stealth}} 

\renewcommand{\setminus}{\smallsetminus}

\DeclareMathOperator{\Hom}{Hom}

\DeclareMathOperator{\Isom}{Isom}
\DeclareMathOperator{\Eq}{Eq}

\DeclareMathOperator{\im}{im}

\newcommand{\Hompin}{\Hom^{\pin}}
\newcommand{\Homchi}{\Hom^{\che}}

\newcommand{\isomto}{\xlongrightarrow{\,\smash{\raisebox{-0.65ex}{\ensuremath{\displaystyle\sim}}}\,}}
\newcommand{\mono}{\hookrightarrow}
\newcommand{\inj}{\mono}

\def\tbp{{\overrightarrow{01}}}
\newcommand{\la}{\longrightarrow}
\DeclareMathOperator*{\colim}{colim}
\DeclareMathOperator{\id}{id}
\DeclareMathOperator{\pr}{pr}
\DeclareMathOperator{\Gal}{Gal}

\DeclareMathOperator{\Pic}{Pic}
\DeclareMathOperator{\Spec}{Spec}

\DeclareMathOperator{\ft}{ft}

\newcommand{\categ}[1]{\textbf{\textup{#1}}}
\newcommand{\Sch}{\categ{Sch}}
\newcommand{\Set}{\categ{Set}}
\newcommand{\Schft}{\Sch^{\ft}}
\newcommand{\Schfft}{\Sch^{\univopen}}

\newcommand{\Schuniawn}{\Sch^{\univopen,\awn}}
\newcommand{\Schsn}{\Sch^{\sn}}
\newcommand{\RTop}{\categ{RTop}}
\newcommand{\RToppin}{\RTop^{\pin}}
\newcommand{\RTopchi}{\RTop^{\che}}

\DeclareMathOperator{\Et}{\mathbf{Et}}

\newcommand{\Scal}{\mathcal{S}}

\newcommand{\Xcal}{\mathcal{X}}
\newcommand{\Ycal}{\mathcal{Y}}

\DeclareMathOperator{\rH}{H}

\newcommand{\bA}{{\mathbb A}}

\newcommand{\bP}{{\mathbb P}}
\newcommand{\bQ}{{\mathbb Q}}

\newcommand{\bZ}{{\mathbb Z}}

\newcommand{\cS}{{\mathscr S}}

\newcommand{\dO}{{\mathcal O}}

\newcommand{\defemph}[1]{\textbf{#1}}

\newcommand{\ab}{{\rm ab}}

\newcommand{\awn}{\text{\rm awn}}
\newcommand{\et}{\text{\rm \'et}} 
\newcommand{\fet}{\text{f\rm \'et}} 

\DeclareMathOperator{\che}{chi}
\DeclareMathOperator{\fcl}{fcl}
\DeclareMathOperator{\fpin}{fpin}

\DeclareMathOperator{\pin}{pin}
\DeclareMathOperator{\redu}{red}
\DeclareMathOperator{\sn}{sn}
\DeclareMathOperator{\univopen}{cuop}

\newcommand{\ph}{\varphi}
\newcommand{\ep}{\varepsilon}
\renewcommand{\epsilon}{\varepsilon}
\renewcommand{\emptyset}{\varnothing}

\DeclareMathOperator{\Jac}{Jac}
\newcommand{\Xet}{X_{\et}}
\newcommand{\Yet}{Y_{\et}}
\DeclareMathOperator{\UH}{UH}
\DeclareMathOperator{\BL}{Bl}

\begin{document}

\title[\Kurztitel]{\Ueberschrift} 

\author{Magnus Carlson}
\address{Magnus Carlson
\newline \indent
Department of Electrical Engineering, Mathematics and Science
\newline \indent
University of Gävle, Sweden}
\email{\tt magnus.carlsson@hig.se}

\author{Jakob Stix}
\address{Jakob Stix
\newline \indent
Institut f\"ur Mathematik, Goethe--Universit\"at Frankfurt
\newline \indent
Robert-Mayer-Str.~6--8, 60325~Frankfurt am Main, Germany}
\email{\tt stix@math.uni-frankfurt.de}

\date{\normalsize \today}

\begin{abstract} 
Let $K$ be a sub-$p$-adic field. We show that the functor sending a finite-type $K$-scheme to its 
\'etale topos is fully faithful after localizing at universal homeomorphisms. This 
generalizes a result of Voevodsky, who proved the analogous theorem for fields 
finitely generated over $\mathbb{Q}$. We also establish a relative version over normal noetherian Nagata schemes with sub-$p$-adic function fields,
where we prove that \emph{componentwise universally open} schemes can be reconstructed from their \'etale topoi, up to universal homeomorphism. Specializing to $\Spec(\bZ)$ as the base, this
partially verifies a conjecture of Grothendieck. Our proofs rely on Mochizuki's $\Hom$-theorem 
in anabelian geometry, a study of point-theoretic morphisms of fundamental groups 
of curves, and a topos-theoretic section conjecture of Voevodsky, which we prove.
\end{abstract}

\maketitle
\setcounter{tocdepth}{1}
{\scriptsize \tableofcontents}
\section{Introduction}
\label{sec:intro}
In his 1983 letter to Faltings \cite{MR1483108}, 
Grothendieck conjectured as one of his anabelian conjectures that,
 for fields $K$ finitely generated over their prime field, \emph{any} scheme $X$ of finite type 
over $K$ can be reconstructed, up to universal homeomorphism, from its \'etale topos $X_{\et}$. In both his letter to Faltings and his \textit{Esquisse d'un Programme} \cite{GrothendieckSketch}, he further conjectured that schemes of finite type over $\Spec(\bZ)$ should be reconstructible, up to universal homeomorphism, from their 
\'etale topos. 
Voevodsky proved in  \cite{MR1098621} that Grothendieck's conjecture is true when $X$ is normal and $K$ is finitely generated of characteristic zero, and Carlson--Haine--Wolf \cite{CHWvoevodskyanabelian}, building on techniques of Voevodsky, proved Grothendieck's conjecture in characteristic zero, and also proved the conjecture in the case when $K$ is a finitely generated field of positive transcendence degree and positive characteristic. 

In this paper, we show that Grothendieck's anabelian conjecture for \'etale topoi holds whenever $K$ is a sub-$p$-adic field, i.e., $K$ is  a subfield of a finitely generated extension of $\mathbb{Q}_p$. This in particular reproves Voe\-vodsky's result. We further prove Grothendieck's conjecture for finite-type schemes 
$X \to \Spec(\bZ)$  
that are \emph{componentwise} universally open: every irreducible component of $X$ is universally open over $\Spec(\bZ)$. More generally, we prove an analogous reconstruction result for higher-dimensional bases.

We now state our results more precisely.  Let $ \Schft_{K}$ be the category of schemes of  finite type over $K$, and let $\Schft_K[\UH^{-1}]$ be the localization along the universal homeomorphisms.  
Write $ \smash{ \RToppin_{K} }$  for the category of coherent topoi over the  \'etale topos $ \Spec(K)_{\et} $ and  
 \defemph{pinned} geometric morphisms (see \cref{subsec:topoipinned}), i.e., morphisms such that the induced map 
$|\Xcal| \to |\Ycal|$ on topological spaces takes closed points to closed points. 

\begin{thmABC}[see \cref{thm:etale_reconstruction}] \label{intro_thm:etale_reconstruction}
	Let $ K $ be a sub-$p$-adic field.
	Then the functor
	\begin{equation*}
		(-)_{\et} \colon \Schft_{K}[\UH^{-1}] \longrightarrow \RToppin_{K}
	\end{equation*}
	sending a scheme $X$ of finite type over $K$ to its \'etale topos $\Xet \to \Spec(K)_{\et}$ is fully faithful.
\end{thmABC}

The localization $\Schft_K[\UH^{-1}]$ has a more concrete description as follows. We denote by $\Schsn_K$ the category of seminormal schemes of finite type over $K$, see \stacks{0EUK}, which in particular contains all normal schemes of finite type over $K$. If $K$ has characteristic zero, then any universal homeomorphism between 
two seminormal schemes $X, Y \in \Schsn_K$
is in fact an isomorphism. 

The inclusion  $\Schsn_K \rightarrow \Sch_K$ admits a right adjoint, the \emph{seminormalization} \stacks{0EUT}, which induces an equivalence $\Schft_K[\UH^{-1}] \rightarrow \Schsn_K$, see  \cite[Theorem~2.7]{CHWvoevodskyanabelian}. Thus, \cref{intro_thm:etale_reconstruction} can be translated as follows. Let $\Hompin_K$ denote the Hom-groupoid  in $\RToppin_K$, see \cref{sec:topoi}, and $\Eq_K$ its full subgroupoid spanned by equivalences. Note that any equivalence between coherent topoi induces a homeomorphism of the underlying topological spaces, in particular, any equivalence is pinned.

\begin{thmABC} \label{intro_thm:etale_reconstructionB}
Let $K$ be a sub-$p$-adic field. Let $X$ and $Y$ be schemes of finite type over $K$.
\begin{enumerate}[(1)]
	\item 
	Suppose $X$ is semi-normal. Then the natural map 
	\[
		(-)_{\et} \colon \Hom_K(X,Y)  \longrightarrow \Hompin_K(X_{\et},Y_{\et})
	\]
	is an equivalence of groupoids. 
	
	\item
	Suppose both $X$ and $Y$ are semi-normal. Then the natural map 
	\[ 
		(-)_{\et} \colon \Isom_K(X,Y) \rightarrow \Eq_K(X_{\et},Y_{\et}) 
	\] 
	is an equivalence of groupoids.
\end{enumerate}
\end{thmABC}

Following  \cite[Proposition~3.10.2]{CHWvoevodskyanabelian}, 
we know that $\Hompin_K(X_{\et},Y_{\et})$ is equivalent to a set, and moreover, we reduce the proof of \cref{intro_thm:etale_reconstructionB} to showing that the natural map 
\[
\Hom_{K}(X,\mathbb{P}_K^1\setminus \{0,1,\infty\} )  \longrightarrow  \Hom_{K}^{\pin}(\Xet,(\mathbb{P}_K^1\setminus \{0,1,\infty\})_{\et})
\]
is bijective whenever $X$ is a smooth, connected, affine $K$-scheme.  The main task is to show that a pinned morphism is either constant, or induces an open map on \'etale fundamental groups. In the latter case, we use Mochizuki's Hom-theorem to conclude \cite[Theorem A]{MochizukiHom}. Generalizing an argument of Creutz and Voloch \cite[Theorem 1.5]{CreutzVolochBrauer},
we establish the dichotomy by proving the following. 

\begin{propABC}[see \cref{prop:hyperbolicopenpi1}] \label{prop:intro}
        Let $X$ and $Y$ be geometrically connected hyperbolic curves over a Kummer-faithful field $K$.
    	Suppose $f \colon   \pi_1(X) \rightarrow \pi_1(Y)$ is a point-theoretic map over $\Gal_K$. 
	Then either $f$ is open, or the image is contained in a single decomposition group.
\end{propABC}

We now state our results over higher-dimensional bases. Given a scheme $S$, we write $\RTopchi_S$ for the category of coherent topoi over the \'etale topos $S_{\et}$ and \defemph{chivalrous} geometric morphisms (see \cref{subsec:topoipinned}): morphisms such that the induced map $|\Xcal| \to |\Ycal|$ of topological spaces takes constructible sets to constructible sets. We denote by $\Schfft_S$ the category of $S$-schemes which are of finite type and componentwise universally open over $S$. We then prove: 

\begin{thmABC}[see \cref{thm:etalereconstructiongeneral} and \cref{prop:localisation-via-awn}] \label{intro_thm:etalereconstructiongeneral}
Let $S$ be a noetherian Nagata scheme that is geometrically unibranch, and assume that the function field\footnote{By the function field of an irreducible component we mean the residue field of its generic point.} of each irreducible component is a sub-$p$-adic field. Then the functor 
\begin{equation*}
		(-)_{\et} \colon \Schfft_{S}[\UH^{-1}] \longrightarrow \RTopchi_{S},
\end{equation*}
sending a componentwise universally open finite-type $S$-scheme $X$  to its \'etale topos $\Xet \to S_{\et}$ over $S_{\et}$, is fully faithful.
\end{thmABC}

\cref{intro_thm:etalereconstructiongeneral} has the following concrete formulation when the source is further normal\footnote{In fact, absolutely weakly normal suffices, but we wanted to keep the statement of \cref{intro_cor:isomhigher} in more classical terms.}.  Note that any equivalence of \'etale topoi over $S$ is automatically chivalrous.

\begin{corABC}[see  \cref{thm:etalereconstructiongeneral}] \label{intro_cor:isomhigher}
Let $S$ be a noetherian Nagata scheme that is geometrically unibranch, and assume that the function field of each irreducible component is a sub-$p$-adic field. Let $X$ and $Y$ be  componentwise universally open $S$-schemes of finite type over $S$.
\begin{enumerate}[(1)]
	\item 
	Suppose $X$ is normal. Then the natural map 
	\[ 
		\Hom_S(X,Y) \to \Hom^{\che}_S(X_{\et},Y_{\et}) 
	\] 
	is an equivalence of groupoids.
	
	\item
	If both $X$ and $Y$ are normal, then the natural map
	\[ 
		\Isom_S(X,Y) \to \Eq_S(X_{\et},Y_{\et}) 
	\] 
	is an equivalence of groupoids.
\end{enumerate}
\end{corABC}

Since flat schemes of finite type over $\Spec(\bZ)$ are componentwise universally open, \cref{intro_cor:isomhigher} trivially implies that if $X$ and $Y$ are normal schemes which are flat and of finite type over $\Spec(\bZ)$, then equivalences of \'etale topoi over $\Spec(\bZ)$ correspond bijectively to isomorphisms of schemes.

To prove \cref{intro_thm:etalereconstructiongeneral}, we first use \cref{intro_thm:etale_reconstructionB} in the generic fiber of $S$. We then spread out the constructed morphism to all of $S$. Crucial in this step is the following proposition, which answers a question of Voevodsky \cite{MR1098621}.

\begin{propABC}[see \cref{prop:voevodskysection}]  \label{introprop:voevodskysection}
Let $S$ be a normal noetherian scheme, and let $X \to S$ be a morphism of finite type. Assume that the function field of each connected component of $S$ is a perfect field. Let $U \to S$ be a quasi-compact separated \'etale morphism. Then every chivalrous map $\sigma \colon U_{\et} \to X_{\et}$ over $S_{\et}$ is induced by a unique morphism of schemes $s \colon U \to X$ over $S$.
\end{propABC}

\addtocontents{toc}{\SkipTocEntry}
\subsection*{Notation and conventions} 

	By an \emph{\'{e}tale cover} we mean a finite \'{e}tale morphism. An \emph{\'{e}tale covering} is a surjective family of \'etale morphisms. 

\subsection*{Acknowledgments}
The first author thanks Eric Ahlqvist, Clark Barwick, Peter Haine, Benjamin Steklov, and Sebastian Wolf
for helpful discussions.
 He especially wishes to thank Peter Haine and Sebastian Wolf, from whom he
 learned a lot on the topic of  \'etale reconstruction, particularly through their joint work 
\cite{CHWvoevodskyanabelian}. Both of the authors
 of this paper gratefully acknowledge support from the 
 Deutsche Forschungsgemeinschaft (DFG) through the Collaborative Research Centre
  TRR 326 ``Geometry and Arithmetic of Uniformized Structures'', project number  444845124.

\section{Preliminaries on topoi}  
\label{sec:topoi}
We recall some well-known results on \'etale topoi of schemes. Much of the material here is already recollected in \cite[\S3]{CHWvoevodskyanabelian}, to which we refer the reader for more details. 

\subsection{Topoi and pinned and chivalrous morphisms} \label{subsec:topoipinned}

Let $\RTop$ denote the $(2,1)$-category of \emph{coherent} topoi with geometric $1$-morphisms\footnote{We work with coherent topoi since the underlying locale is \emph{spatial}, that is, represented by a topological space, which is moreover spectral. This is convenient when we define \emph{pinned} and \emph{chivalrous} geometric morphisms.}. We recall the relative notion of the $(2,1)$-category $\RTop_{\Scal}$ of topoi sliced over a particular coherent topos $\Scal$:
\begin{enumerate}[(i)] 
	\item 
		Objects of $\RTop_{\Scal}$ are geometric $1$-morphisms of topoi $p_\ast \colon \Xcal \to \Scal$.
	\item 
		Given $p_\ast \colon  \Xcal \to \Scal$  and $q_\ast \colon \Ycal \to \Scal$, the Hom-groupoid  
		$\Hom_{\Scal}\big(p_\ast,q_\ast\big)$ has as objects pairs $(\ph_\ast,\alpha)$ where 
		$\ph_\ast \colon \Xcal \to \Ycal$ is a geometric morphism and 
		$\alpha \colon {q_\ast \ph_\ast} \isomto {p_\ast}$ is a natural isomorphism. The morphisms 
		from $ (\ph_\ast,\alpha) \colon  \Xcal \to \Ycal$ to $(\psi_\ast,\beta) \colon \Xcal \to \Ycal$
		 are the natural isomorphisms $\vartheta \colon \ph_\ast \isomto \psi_\ast $ such that 
		\begin{equation*} 
			\begin{tikzcd}
			  & q_\ast \psi_\ast \arrow[dr," \beta"] \\
			 q_\ast \ph_\ast \arrow[rr, "\alpha"'] \arrow[ur,"q_\ast \vartheta"] & & p_\ast
			\end{tikzcd} 
		\end{equation*}
		commutes.
\end{enumerate}
For an object $(\ph_\ast,\alpha) \in \Hom_{\Scal}\big(p_\ast,q_\ast\big)$, we will most often write $\varphi$, and leave $\alpha$ implicit. We will be explicit when $\alpha$ is needed in our arguments.

To a topos $\Xcal$ one can functorially associate a locale $|\Xcal|$ by considering 
the locale of subobjects of the terminal object \cite[\S3.4]{CHWvoevodskyanabelian}.  
When $\Xcal$ is a \emph{coherent} topos, the locale $|\Xcal|$ comes from a topological space,
which is moreover spectral \cite[Chapter II, \S 3.3-3.4]{JohnstoneStone}. Given a map $\ph \colon  \Xcal \rightarrow \Ycal$ of coherent topoi, the natural
map $|\ph| \colon |\Xcal| \rightarrow |\Ycal|$ is continuous. Call a geometric morphism of coherent topoi
$ \ph \colon  \Xcal \rightarrow \Ycal$ \defemph{pinned} if the underlying map of topological spaces
takes closed points to closed points. Given coherent topoi $ \Xcal,\Ycal \in \RTop_{\Scal} $, write
\begin{equation*}
 	\Hompin_{\Scal}(\Xcal,\Ycal) \subseteq \Hom_{\Scal}(\Xcal,\Ycal)
\end{equation*}
for the full subgroupoid spanned by the pinned geometric morphisms. Define
$\smash{ \RToppin_{\Scal} }$  as the $(2,1)$-category with the same objects as 
$\RTop_{\Scal}$,  pinned geometric
morphisms over $\Scal$ as $1$-morphisms and natural isomorphisms between pinned geometric morphisms over $\Scal$ as $2$-morphisms. Similarly, a geometric morphism $\ph \colon \Xcal \rightarrow \Ycal$ of coherent topoi is
\defemph{chivalrous} if the induced map of topological spaces sends constructible subsets to constructible subsets. We define $\smash{\RTopchi_{\Scal}}$ to be the $(2,1)$-category with the same objects as $\RTop_{\Scal}$, with $1$-morphisms chivalrous geometric morphisms over $\Scal$, and natural isomorphisms between $1$-morphisms as $2$-morphisms. Given coherent topoi $p_\ast  \colon \Xcal \to \Scal$ and $q_\ast \colon \Ycal \to \Scal$ over $\Scal$, we write 
\[ 
	\Homchi_{\Scal}(\Xcal,\Ycal) \subseteq \Hom_{\Scal}(\Xcal,\Ycal) 
\] 
for the full subgroupoid spanned by the chivalrous geometric morphisms.

For us, $\Scal$ will always be the \'etale topos $\Scal =  S_{\et}$ of a qcqs scheme $S$. The \'etale topos of any qcqs scheme is coherent. If $\Scal = S_{\et}$, we write 
\[ 
	\RTop_S = \RTop_{S_{\et}} \qquad  \text{ and } \qquad \Hom_S(\Xcal,\Ycal) = \Hom_{S_{\et}}(\Xcal,\Ycal). 
\] 
When $S = \Spec(R)$ is affine, we further use the shorthand 
\begin{equation*}
	\RTop_{R} = \RTop_{\Spec(R)_{\et}} 
	\qquad  \text{ and } \qquad  
	\Hom_{R}(\Xcal,\Ycal) = \Hom_{\Spec(R)_{\et}}(\Xcal,\Ycal) 
\end{equation*} 
and similar notation for pinned and chivalrous variants of the above.

\begin{remark} \label{rmk:basechangechivalrous}
Let $S$ be a qcqs scheme and let $X$ and $Y$ be qcqs schemes locally of finite type over $S$. Note that the structure maps $X \to S$ and $Y \to S$ are quasi-compact by  \stacks{03GI}, thus $X_s$ and $Y_s$ are quasi-compact for every $s \in S$. For a chivalrous map $\ph \colon \Xet \to \Yet$ over $S_{\et}$ and any point $s \in S$, \cite[Lemma~3.6]{CHWvoevodskyanabelian} shows that we obtain a geometric morphism
\[
\ph_s \colon (X_s)_{\et} \to (Y_s)_{\et}
\]
over $k(s)_{\et}$, which we claim is pinned. Let $x \in X_s$ be a closed point, and let $U_s \subseteq |X_s|$ be the open complement of $\{x\}$ in $X_s$. Since $X_s$ is locally of finite type over $\Spec k(s)$ and is further quasi-compact,  $X_s$ is noetherian. Thus, $U_s$ is quasi-compact; choose an open subset $W \subseteq X$ such that $W \cap |X_s| = U_s$. For every $z \in U_s$, choose an affine open $V_z \subseteq W$ of $X$ which contains $z$. Note that $V_z$ is retrocompact in $X$ since $X$ is quasi-separated \stacks{054D}. Since $U_s$ is quasi-compact, finitely many $V_z$ cover~$U_s$. Let $V$ be their union and $C = |X| \setminus V$. By definition, $C$ is constructible in $X$, and $C \cap |X_s| = \{x\}$. By assumption, $|\ph|(C)$ is constructible in $Y$, and 
\[
|\ph|(C) \cap |Y_s| = \{ |\ph_s|(x) \}
\]
is constructible in $Y_s$.  Since $Y_{s}$ is locally of finite type over $\Spec(k(s))$ and quasi-compact, it is a noetherian Jacobson scheme. For a noetherian Jacobson scheme, a point is constructible if and only if it is closed \stacks{005Z}. Thus $\ph_s$ is pinned.
\end{remark}

\begin{remark} \label{rmk:equivalenttoset}
As is shown in  \cite[Proposition~3.10]{CHWvoevodskyanabelian},  for $X$ and $Y$ locally 
of finite type over a scheme $S$, the groupoid $\Hompin_{S}(\Xet,\Yet)$ is equivalent to a set.
Now let $X$ and $Y$ be schemes of finite type over $S$, and assume moreover that $S$ is noetherian. Using \cref{rmk:basechangechivalrous}, the proof of \cite[Proposition~3.10]{CHWvoevodskyanabelian} adapts to show that $\Homchi_{S}(\Xet,\Yet)$ is equivalent to a set.
Thus, the $(2,1)$-categories spanned by pinned and chivalrous geometric morphisms between  \'etale topoi of 
schemes of finite type over a noetherian scheme $S$ are equivalent to $1$-categories.
\end{remark}

\begin{recollection} \label{rec:sitesandtopoi}
 Let $S$ be a noetherian scheme, let $X$ and $Y$ be schemes of finite type over $S$. Let $\ph \colon \Xet \to \Yet$ 
be a geometric morphism over $S_{\et}$. Then, for every quasi-compact separated \'etale morphism $V \to Y$, the pullback $\varphi^* V$ is represented by a quasi-compact separated \'etale morphism $\varphi^* V \to X$ \cite[Theorem A.8, Corollary A.9]{CHWvoevodskyanabelian}. 
We write $\Et_X$ and $\Et_Y$ for the \'etale sites whose objects are quasi-compact separated \'etale morphisms over $X$ and $Y$ respectively. Then $\varphi$ determines a morphism of sites $\varphi^* \colon \Et_X \to \Et_Y$ (which is a suitable functor in the opposite direction). The $S$-analogue of \cite[Corollary A.11]{CHWvoevodskyanabelian} gives an equivalence of groupoids 
\[ 
	\Hom_S(\Et_X,\Et_Y) \isomto \Hom_S(X_{\et},Y_{\et}). 
\] 
We say that a morphism of sites $\varphi^* \colon \Et_X \to \Et_Y$ is chivalrous if the induced geometric morphism is chivalrous. By restricting, the equivalence just established gives us an equivalence 
\[
\Hom^{\che}_S(\Et_X,\Et_Y) \isomto  \Hom^{\che}_S(X_{\et},Y_{\et}).
\]
\end{recollection}

\subsection{Base points of topoi}

Recall \cite[Section 8.4]{JohnstoneTopostheory} that given a connected topos $\mathcal{X}$ and a point 
$p \colon  \Set \rightarrow \mathcal{X}$, one has a Galois category 
$\mathcal{X}_{\fet}$, the full subcategory of $\mathcal{X}$ 
 consisting of objects of $\mathcal{X}$ 
 which are locally constant constructible. A  fiber functor is given by
 $p^\ast \colon  \mathcal{X}_{\fet} \rightarrow \Set$. Further, a geometric morphism 
 $\ph \colon  \mathcal{X} \rightarrow \mathcal{Y}$ gives rise to a map 
 $\ph^\ast \colon  (\mathcal{Y}_{\fet}, p^\ast \circ \ph^\ast) \rightarrow (\mathcal{X}_{\fet},p^\ast)$ of 
 Galois categories. In the case that $\Xet$ and $\Yet$ are  \'etale topoi of connected schemes 
 and $\ph \colon  \Xet \rightarrow \Yet$ is a geometric morphism of \'etale topoi, and $\overline{x}$ is a basepoint
 for $X$, one gets, by considering automorphism groups 
 of fiber functors, a map 
 \[
 	\ph_\ast \colon  \pi_1(X,\overline{x}) \rightarrow \pi_1(Y, \ph \circ \overline{x})
 \]
 of \'etale fundamental groups.
 
\begin{recollection} \label{rec:pinnedetale}
Suppose that $X$ and $Y$ are of finite type over a perfect field $K$ and that $\ph \colon  \Xet \rightarrow \Yet$
is a pinned morphism of \'etale topoi over $K_{\et}$. Then for any algebraic extension $L/K$ and any morphism
$x\colon\Spec(L) \rightarrow X$ over $K$, the composite 
\[
	\ph \circ x_{\et} \colon L_{\et}  \longrightarrow  Y_{\et}
\]
comes from a unique morphism of schemes $\Spec(L) \rightarrow Y$  over $K$ 
\cite[Proposition~4.5]{CHWvoevodskyanabelian}. This implies that we have, for every  algebraic extension $L/K$, a map $\ph(L) \colon X(L) \rightarrow Y(L)$.  In the colimit over all $L$ we obtain a unique map 
\[
	\ph(\overline{K}) \colon X(\overline{K}) \longrightarrow Y(\overline{K})
\]
such that for all points $a \in X(L) \subseteq X(\overline{K})$ we have $\ph \circ a_{\et} = b_{\et}$ for $b = \ph(L)(a)$. 

Further, for any quasi-compact and separated \'etale open $j \colon V \rightarrow Y$, we 
know by \cite[Theorem A.8]{CHWvoevodskyanabelian} that $\ph^*V \in \Xet$ is represented by a quasi-compact 
and separated  \'etale open $j'  \colon  \ph^* V \rightarrow X$. By taking slice topoi, we get a geometric morphism $(\ph^*V)_{\et} \rightarrow V_{\et}$, the restriction of $\ph$, which is again pinned. This 
follows from the fact that $\ph$, $j_{\et}$ and $j'_{\et}$ are pinned, and the fact that, 
since $j: V \rightarrow Y$ is  \'etale and quasi-compact, it is quasi-finite. 
\end{recollection}

\begin{remark} \label{rmk:chivalrousrestriction}
Suppose that $X$ and $Y$ are qcqs schemes and that $\ph \colon X_{\et} \to Y_{\et}$ is a \emph{coherent} geometric morphism which is further chivalrous. Note that by \cite[Lemma A.7]{CHWvoevodskyanabelian}, $\ph$ is coherent if and only if the induced map 
\[ 
	|\ph| \colon |X| \to |Y| 
\] 
between underlying topological spaces is quasi-compact. Then for any quasi-compact and separated \'etale open $V \rightarrow Y$, the pullback $\ph^*V \in \Xet$ is represented by a quasi-compact and separated \'etale open $\ph^* V \rightarrow X$ \cite[Theorem A.8]{CHWvoevodskyanabelian}.  The morphism $(\ph^* V)_{\et} \to X_{\et}$ is chivalrous by Chevalley's theorem  
\stacks{054K}, hence so is $(\ph^* V)_{\et} \to Y_{\et}$ as the composition of chivalrous maps. In contrast, it does not follow formally that the geometric morphism $(\ph^*V)_{\et} \to V_{\et}$ obtained by taking slice topoi is chivalrous.
\end{remark}

\section{Morphisms between  \'etale topoi and fundamental groups} \label{sec:pi1}

The purpose of this section is to show that given a non-constant $K_\et$-morphism
$\ph  \colon   U_{\et} \rightarrow V_{\et}$  between the \'etale topoi of smooth geometrically connected curves $U$ and $V$ over a Kummer-faithful field $K$, see \cref{defi:Kummer_faithful}, that the induced map of fundamental groups
 $\ph_\ast \colon \pi_1(U,\bar u) \rightarrow \pi_1(V,\bar v)$ is open, where $\bar v = \ph(\bar u)$, 
 see  \cref{prop:hyperbolicopen}.

Throughout \cref{sec:pi1}, we fix a field $K$ with a fixed algebraic closure $\overline{K}$, i.e., a base point for $\Spec(K)$. As soon as we have defined the notion of a Kummer-faithful field, the field $K$ will be assumed to be Kummer-faithful. We will further assume for the main result of this section that $K$ has characteristic~$0$.

 \subsection{Quasi-sections and point-theoretic maps} \label{subseq:quasi_sections}

Let $X/K$ be a geometrically connected scheme of finite type
and endow $X$ with a geometric point $\bar x$ compatible with $\overline K$. For any finite extension $L/K$ inside $\overline{K}$ the group $\Gal_L = \Gal(\overline{K}/L)$ is an open subgroup of $\Gal_K = \Gal(\overline{K}/K)$.  We denote by 
\[
\cS_{X/K}(L) \coloneq \{ s\colon \Gal_L \to \pi_1(X,\bar x) \ ; \text{ over } \Gal_K\}/_{\pi_1(X_{\bar K},\bar x)}
\]
the set of $\pi_1(X_{\bar K},\bar x)$ conjugacy classes of $L$-\defemph{rational quasi-sections} (i.e., sections only defined on the open subgroup $\Gal_L$). Functoriality of the \'etale fundamental group yields the non-abelian Kummer map
\[
X(L) \longrightarrow \cS_{X/K}(L), \qquad a \mapsto \big(\pi_1(a)\colon\Gal_L \to \pi_1(X,\bar x) \big).
\]
We may pass to the limit over all $L \subseteq \overline K$ that are finite over $K$ and obtain
\[
\kappa\colon X(\overline{K}) \longrightarrow \cS_{X/K}(\overline{K}) \coloneq  \colim_{L}  \cS_{X/K}(L),
\]
the non-abelian Kummer map from $\overline{K}$-rational points to the set of \defemph{quasi-sections}. The quasi-sections in the image of $\kappa$ are called \defemph{geometric quasi-sections}. 

\begin{remarks}
\begin{enumerate}
	\item
	Any $L$-rational geometric quasi-section defines a $\pi_1(X_{\bar K},\bar x)$-conjugacy class of closed subgroups of $\pi_1(X,\bar x)$ as the image of the section. The maximal subgroups among these images are nothing but the (conjugacy classes of) decomposition subgroups associated to closed points of $X$ (or rather some choice of universal pro-cover).
	\item
	Note that the sets of quasi-sections $\cS_{X/K}(L)$ and $\cS_{X/K}(\overline{K})$ do not depend on the choice of a base point of $X$ up to canonical and compatible bijections.
\end{enumerate}
\end{remarks}

\begin{definition}
Let $X$ and $Y$ be geometrically connected schemes of finite type over $K$. A homomorphism $\psi \colon  \pi_1(X,\bar x) \to \pi_1(Y,\bar y)$ over $\Gal_K$ is called \defemph{point-theoretic} if composition with $\psi$ preserves geometric quasi-sections. In other words, for all finite $L/K$ inside $\overline{K}$ and all points $a \in X(L)$ there is a finite extension $L'/L$ inside $\overline{K}$ and a point $b \in Y(L')$ such that $\pi_1(b) = (\psi \circ \pi_1(a))|_{\Gal_{L'}}$.
\end{definition}

\begin{lemma} \label{lemma:pinned_yields_pointtheoretic}
Let $\ph   \colon   X_{\et} \rightarrow Y_{\et}$ be a pinned morphism of topoi over $K_{\et}$
between two geometrically connected varieties over $K$. Then the induced map 
$\ph_\ast   \colon   \pi_1(X,\bar x) \rightarrow \pi_1(Y,\bar y)$ is point-theoretic. 
\end{lemma}

\begin{proof}
Indeed, by  \cref{rec:pinnedetale}, for every finite extension $L/K$ and every point $a \in X(L)$ the composite
\[
	\ph \circ a_{\et} \colon (\Spec(L))_{\et} \longrightarrow X_{\et} \longrightarrow Y_{\et}
\]
is induced by a morphism of schemes $b \colon  \Spec(L) \rightarrow Y$. Hence also 
\[
	\ph_\ast \circ \pi_1(a) = (\ph \circ a_{\et})_\ast = (b_{\et})_\ast = \pi_1(b). \qedhere
\]
\end{proof}

 \subsection{Galois sections of the generalized Jacobian} \label{subseq:jacobian}
In this subsection we ask the field $K$ to be of characteristic zero. We extend a construction proposed in \cite[\S13.5]{StixRational} only in the case of smooth projective curves. 

\begin{remarks}
\begin{enumerate}
	\item
	If $X$ is a disjoint union $X = \amalg_i X_i$ of geometrically connected schemes of finite type $X_i$,
	then we can extend the definition of the sets of geometric quasi-sections as 
	\[
	\cS_{X/K}(L) \coloneq \amalg_i \cS_{X_i/K}(L)
	\]
	and similarly with coefficients in $\overline{K}$.
	
	\item
	Let $X$ and $Y$ be geometrically connected schemes of finite type over $K$. 
	The K\"unneth-formula yields an isomorphism
	\[
	(\pr_{1\,\ast}, \pr_{2,\ast}) \colon \pi_1(X \times_K Y,(\bar x, \bar y)) 
	\isomto  
	\pi_1(X,\bar x) \times_{\Gal_K}  \pi_1(Y,\bar y).
	\]
	It follows that the projection maps induce canonical bijections
	\[
	\cS_{X\times_K Y/K}(L) \isomto \cS_{X/K}(L) \times \cS_{Y/K}(L),
	\]
	and similarly with coefficients in $\overline{K}$.	
\end{enumerate}
\end{remarks}

\begin{recollection} \label{recollection:genjac}
Let $U$ be a smooth geometrically connected curve over $K$ with smooth projective completion $X /K$. The generalized Picard scheme $\Jac^\bullet_U \coloneq \Pic_{X,D}$ parametrizes line bundles on $X$, together with a trivialization along $D = X \setminus U$. The connected components of $\Jac^\bullet_U$ are the subschemes $\Jac_U^d$ of line bundles of degree $d \in \bZ$.  For $d=0$ we recover the generalized Jacobian $\Jac_U = \Jac^0_U$ of $U$.

Tensor product of line bundles defines an abelian algebraic group structure on $\Jac^\bullet_U$. Therefore 
\[
	\cS_{\Jac^\bullet_U/K}(\overline{K})
\]
is an abelian group. The degree map defines a short exact sequence
\[
	0 \to \cS_{\Jac_U}(\overline{K}) \to \cS_{\Jac^\bullet_U/K}(\overline{K}) \xrightarrow{\deg} \bZ \to 0.
\]
For all integers $d$ taking $d$-th tensor powers yields a map $\Jac^1_{U} \to \Jac^d_{U}$ that induces an isomorphism  
\[
	d_\ast \big(\pi_1(\Jac^1_{U})\big) \isomto \pi_1(\Jac^d_{U}) 
\]
over $\Gal_K$, where $d_\ast \big(\pi_1(\Jac^1_{U})\big)$ is the pushout of the extension
\[
	1 \to  \pi_1((\Jac^1_{U})_{\bar K}) \to  \pi_1(\Jac^1_{U})  \to \Gal_K \to 1
\]
along the multiplication by $d$ map on the abelian group $\pi_1((\Jac^1_{U})_{\bar K})$.  Note that for $d=0$, the resulting base point in $\Jac_U^0$ is $0$, and the extension of fundamental groups
is canonically split.
\end{recollection}

\begin{recollection} \label{recollection:albanese}
The generalized Abel-Jacobi map,
\[
   	 j_U  \colon  U \longrightarrow \Jac^1_{U}, \qquad P \mapsto \dO_X(P),
\]   
is the universal map into a torsor under a semi-abelian variety, the semi-abelian Albanese morphism. 
The non-abelian Kummer map is compatible with the Abel-Jacobi map. It moreover factors over the group of $0$-cycles on $U_{\bar K}$ as follows:
\[
	\xymatrix@M+1ex{
	U(\overline{K}) \ar[d]^{\kappa} \ar[r] & 
	Z_0(U_{\bar K}) \ar[r]^{j_U} & 
	\Jac^\bullet_U(\overline K) \ar[d]^{\kappa} \\
	\cS_{U/K}(\overline{K})  \ar[rr]^{j_{U,\ast}}  && 
	\cS_{\Jac^\bullet_U/K}(\overline{K}) }
\]
    
From \cite[Section II]{KatzLangFiniteness} we deduce that $\pi_1(j_U)$ induces an isomorphism
\[
     \pi^{(\ab)}_1(U,\bar u) \rightarrow \pi_1(\Jac^1_{U},\bar u_1),
\]
where $\pi_1^{(\ab)}(U,\bar u)$ denotes the geometric abelianization of $\pi_1(U,\bar u)$, and $\bar u_1 = j_U(\bar u)$. It follows from the description of $\pi_1(\Jac^d_{U})$ given above, that for all integers $d$ the extension $\pi_1(\Jac^d_{U})$ can be reconstructed from $\pi_1(U,\bar u)$. This proves the following lemma.
\end{recollection}

\begin{lemma} \label{lem:effect on abelianized sections}
Let $U, V$ be smooth, geometrically connected curves over $K$, and let 
\[
\psi \colon  \pi_1(U,\bar u) \to \pi_1(V,\bar v)
\]
be a homomorphism over $\Gal_K$. Then there are unique group homomorphisms
\[
\psi^{\ab,d} \colon  \pi_1(\Jac^d_U,\bar u_d) \to \pi_1(\Jac^d_V,\bar v_d)
\]
over $\Gal_K$, where $\bar u_d = d (j_U(\bar u))$ and similarly $\bar v_d$ such that
\[
\psi^{\ab,1} \circ \pi_1(j_U) = \pi_1(j_V) \circ \psi
\]
and the $\psi^{\ab,d}$'s are compatible with the maps induced by multiplication on $\Jac^\bullet_U$ and $\Jac^\bullet_V$.

In particular, we obtain a commutative diagram
\[
\xymatrix@M+1ex{
\cS_{U/K}(\overline{K}) \ar[r]^{\psi}  \ar[d]^{j_{U,\ast}} & 
\cS_{V/K}(\overline{K}) \ar[d]^{j_{V,\ast}} \\
\cS_{\Jac^\bullet_U/K}(\overline{K}) \ar[r]^{\psi^{\ab,\bullet}}  & 
\cS_{\Jac^\bullet_V/K}(\overline{K})}
\]
in which the bottom map is a group homomorphism.
\end{lemma} 

\subsection{Sections over Kummer-faithful fields} \label{subseq:sectionskummer}
We recall the definition \cite[Definition 1.5]{MR3445958} for characteristic $0$,  see \cite[Definition 1.2]{Hoshi:KummerFaithful} for a version over perfect fields.
\begin{definition}
\label{defi:Kummer_faithful}
A  \defemph{Kummer-faithful} field in characteristic $0$ is a field $K$ of characteristic $0$ such that for all semi-abelian varieties $B/K$ the following equivalent conditions hold.
\begin{enumerate}[(a)]
	\item 
	The Kummer map $B(K) \to \rH^1(K,\pi_1^\ab(B_{\bar K}))$ is injective.
	\item
	The intersection $\bigcap_{n \geq 1} n B(K) = 0$ is trivial, i.e., $B(K)$ does not contain arbitrarily divisible elements. 
\end{enumerate}
The equivalence is a consequence of the exact sequence of Kummer theory.
\end{definition}

\begin{remarks} \label{rmk:sorite_Kummerfaithful}
\begin{enumerate}
	\item
	Note that the Kummer map of Kummer theory applied to $B/K$ agrees (by comparing with the $0$-section) with the non-abelian Kummer map of the section conjecture 
	\[
	B(K) \to \cS_{B/K}(K) = \rH^1(K, \pi_1(B_{\bar K})) = \rH^1(K,\pi_1^\ab(B_{\bar K})),
	\]
	see for example \cite[Corollary 71]{StixRational}.
	\item
	Sub-$p$-adic fields are Kummer-faithful, but not conversely, see for example 
	\cite[Remark 1.5.4]{MR3445958}, \cite[Remark 1.2.3]{Hoshi:KummerFaithful} and \cites{OhtaniKummer,MR4590273}.
	\item
	Weil restriction shows that finite extensions of Kummer-faithful fields are Kummer-faithful.
\end{enumerate}
\end{remarks}

From now on $K$ is assumed to be a Kummer-faithful field. 

\begin{lemma}
\label{lem:kummermapsubsemiabelian}
Let 
$X/K$ be a geometrically connected scheme of finite type that admits an injective map $X \to W$ into a torsor $W/K$ under a semiabelian variety $B/K$.  Then for any finite extension $L/K$ (inside $\overline K$) the non-abelian Kummer map 
\[
\kappa \colon X(L)  \longrightarrow \cS_{X/K}(L)
\]
is injective and induces a bijection of $X(L)$ with the set of geometric $L$-rational quasi-sections.
\end{lemma}
\begin{proof}
This follows at once from the naturality of the non-abelian Kummer map, and the semi-abelian analogue of \cite[Proposition 73]{StixRational}. 
\end{proof}

The proof of the following is immediate in view of \cref{lem:kummermapsubsemiabelian}.

\begin{corollary} \label{cor:pointtheoretic_actualmap}
Let 
$X$ and $Y$ be geometrically connected schemes of finite type over $K$ such that $Y$ 
admits an injective map $Y \to W$ into a torsor $W/K$ under a semiabelian variety $B/K$. 
Then for any point-theoretic homomorphism $\psi \colon  \pi_1(X,\bar x) \to \pi_1(Y,\bar y)$ over $\Gal_K$ there is a unique map of sets 
\[
\psi(\overline{K}) \colon X(\overline{K}) \longrightarrow Y(\overline{K}),
\]
which is  the colimit of the maps $\psi(L) \colon X(L) \longrightarrow Y(L)$ 
defined for all $a \in X(L)$ by the equality $\pi_1(b) = \psi \circ \pi_1(a)$ for $b = \psi(L)(a)$, i.e., $\psi(\overline{K})$ agrees with $\psi \circ - $ on geometric quasi-sections.
\end{corollary}

\begin{proposition} \label{prop:pinnedpoint}
Let 
$X$ and $Y$ be geometrically connected schemes of finite type over $K$ such that $Y$ 
admits an injective map $Y \to W$ into a torsor $W/K$ under a semiabelian variety $B/K$. 
Let $\ph \colon X_{\et} \rightarrow Y_{\et}$ be a pinned morphism of topoi over $K$. Then the two maps
\[
\ph(\overline{K}), \ph_\ast(\overline{K})  \colon  X(\overline{K}) \longrightarrow Y(\overline{K})
\]
agree, where $\ph_\ast \colon \pi_1(X,\bar x) \to \pi_1(Y,\bar y)$ is the map induced by $\ph$, and $\bar y = \ph \circ \bar x$.
\end{proposition} 
\begin{proof}
By \cref{lemma:pinned_yields_pointtheoretic} the homomorphism $\ph_\ast$ is point-theoretic, so that $\ph_\ast(\overline{K})$ is well defined by \cref{cor:pointtheoretic_actualmap}. The proof of \cref{lemma:pinned_yields_pointtheoretic} actually shows directly the apparently stronger statement of the claim in view of the functoriality of the \'etale fundamental group.
\end{proof}

 \subsection{Point-theoretic maps and the Abel-Jacobi map} \label{subseq:AJ}
We remind the reader that we assume that $K$ is a Kummer-faithful field of characteristic $0$. 

\begin{lemma} \label{lemma:albanese}
Let $U$ be a geometrically connected smooth curve over $K$.
Then the non-abelian Kummer map 
\[
\kappa \colon \Jac^\bullet_{U}(\overline{K}) \longrightarrow  \cS_{\Jac^\bullet_U/K}(\overline K) 
\]
is injective.
\end{lemma}

\begin{proof}
This follows from the definition of Kummer-faithful in combination with 
\cref{rmk:sorite_Kummerfaithful} since the generalized Jacobian $\Jac_U^0$ is a semi-abelian variety over $K$.
\end{proof}

\begin{proposition} \label{prop:wellbehavedinducespic}
Let $U$ and $V$ be smooth, geometrically connected curves over $K$,
and suppose that $\pi_1(V_{\overline{K}},\bar v)$ is non-trivial. 
Given a point-theoretic morphism 
\[
\psi  \colon  \pi_1(U, \bar u) \rightarrow \pi_1(V, \bar v)
\]
over $\Gal_K$, then there is a unique group homomorphism 
\[
\psi^{\ab}(\overline{K})   \colon   \Jac^\bullet_{U}(\overline{K}) \rightarrow \Jac^\bullet_{V}(\overline{K})
\]
which makes the diagram
\[
\begin{tikzcd}
U(\overline{K})  \arrow[r,"\psi(\overline{K})"] \arrow[d,"j_U"] & V(\overline{K}) \arrow[d,"j_V"] \\ 
 \Jac^\bullet_{U}(\overline{K})  \arrow[r,"\psi^{\ab}(\overline{K})"] & \Jac^\bullet_{V}(\overline{K})
\end{tikzcd}
\]
commutative.
\end{proposition}

\begin{proof}
From \cref{recollection:albanese} we know that the top and bottom parts commute in the diagram
\[
\xymatrix@M+0.5ex@R-2ex@C-1ex{
U(\overline{K}) \ar[dr]^{\kappa} \ar[rr] \ar[dd]^{\psi(\overline{K})} &&
Z_0(U_{\bar K}) \ar@{->>}[rr]^{j_U} \ar[dd]|!{[dl];[dr]}\hole^(0.65){\bZ[\psi(\overline{K})]} &&  
\Jac^\bullet_U(\overline K) \ar[dr]^{\kappa}  \ar@{.>}[dd]|!{[dl];[dr]}\hole^(0.65){\psi^{\ab}(\overline{K})}  & \\
& \cS_{U/K}(\overline{K})  \ar[rrrr]^{j_{U,\ast}}  \ar[dd]^(0.65){\psi} &&&& 
\cS_{\Jac^\bullet_U/K}(\overline{K})  \ar[dd]^{\psi^{\ab,\bullet}} \ 
\\
V(\overline{K}) \ar[dr]^{\kappa} \ar[rr]|!{[dr];[ur]}\hole &&
Z_0(V_{\bar K}) \ar[rr]^{j_V} &&  
\Jac^\bullet_V(\overline K)  \ar@{^(->}[dr]^{\kappa}  & \\
& \cS_{V/K}(\overline{K})  \ar[rrrr]^{j_{V,\ast}}  &&&& 
\cS_{\Jac^\bullet_V/K}(\overline{K}) \ .
}
\]
As $\pi_1(V_{\overline{K}})$ is non-trivial by assumption, the map $j_V \colon V \to \Jac^1_V$ is injective. Therefore, since $\psi$ is point-theoretic and $K$ is Kummer-faithful, \cref{cor:pointtheoretic_actualmap} provides the map $\psi(\overline{K})$ such that the left face commutes.
  
The map $\psi^{\ab,\bullet}$ was constructed in \cref{lem:effect on abelianized sections} such that the front face commutes.

The map $\bZ[\psi(\overline{K})]$ is the unique group homomorphism extending the map $\psi(\overline{K})$. 
By \cite[Chapter V, Proposition 3]{SerreClass}, the natural map 
$j_U \colon Z_0(U_{\overline{K}})  \rightarrow \Jac^\bullet_{U}(\overline{K}) $ is surjective. Furthermore, the bottom map $\kappa$ on the right-hand side is injective by \cref{lemma:albanese}. Thus a simple diagram chase, using that $Z_0(U_{\bar K})$ is the free abelian group on $U(\overline{K})$, shows the existence and uniqueness of the claimed homomorphism $\psi^{\ab}(\overline{K})$.
\end{proof}

We will now essentially follow the argument of  Creutz and Voloch \cite[Theorem 1.5]{CreutzVolochBrauer} 
to show that any point-theoretic map $\psi \colon \pi_1(U, \bar u) \rightarrow \pi_1(V, \bar v)$
of hyperbolic curves is either open, or the image of $\psi$ is contained in a single decomposition group.

For a geometrically connected curve $U$ over $K$ with smooth projective compactification $X$, 
let $g_U$ be the (geometric) genus of $X$, and 
let $n_U  = \deg(X_{\overline{K}} \setminus U_{\overline{K}})$ be the degree of the boundary divisor. 
Finally, set $\epsilon_U =1$ if $U$ is affine and $0$ otherwise. 
The $\ell$-adic Euler characteristic of $U_{\bar K}$ is
\[
\chi_U \coloneq \chi(U_{\bar K}, \bQ_\ell) = 2-2g_U-n_U.
\]
By  \cite[Chapter V,Theorem 1]{SerreClass}, the dimension of $\Jac^0_U$ is 
\begin{equation}
\label{eq:dimJac}
r_U \coloneq \dim \Jac^0_U = g_U+n_U - \ep_U = 1-\ep_U + \frac{1}{2}(n_U - \chi_U).
\end{equation}

\begin{proposition}	 \label{prop:eulerchar}
Let  $U$ and $V$ be geometrically connected smooth curves over a Kummer-faithful field $K$ of characteristic zero,  and assume that $\pi_1(V_{\overline{K}},\bar v)$ is non-trivial. Let  
\[
\psi \colon \pi_1(U,\bar u) \rightarrow \pi_1(V,\bar v)
\]
be a point-theoretic morphism over $\Gal_K$. 

Then, if $\chi_V \leq - 4 g_U  - 4 n_U$,  the image of  $\psi(\overline{K}) \colon U(\overline{K}) \rightarrow V(\overline{K})$ is finite.
\end{proposition}

\begin{proof}
We use \eqref{eq:dimJac} to translate the numerical assumption into
\[
2r_U - r_V = 2(g_U+n_U - \ep_U)  + \frac{1}{2}( \chi_V -n_V) + \ep_V - 1  \leq - 2\ep_U - \frac{1}{2}n_V + \ep_V -1 < 0,
\]
so that the estimate $2 r_U < r_V$ holds.

We may replace the base points by $\bar u \in U(\overline{K})$ and $\bar v = \psi(\overline{K})(\bar u) \in V(\overline{K})$. Moreover, we consider the Abel-Jacobi maps with respect to these base points
\[
j_{\bar u} : U_{\bar K}  \to \Jac^0_{U,\bar K}, \quad z \mapsto j_U(z) - \bar u \quad \text{ and } \quad j_{\bar v} : V_{\bar K} \to \Jac^0_{V,\bar K}, \quad z \mapsto j_V(z) - \bar v .
\]
Then the diagram of \cref{prop:wellbehavedinducespic} gives rise to a commutative diagram
\[
\begin{tikzcd} 
U(\overline{K}) \arrow[r,"\psi(\overline{K})"] \arrow[d,"j_x"] & 
V(\overline{K}) \arrow[d,"j_y"] 
 \\  
 \Jac^0_{U}(\overline{K}) \arrow[r,"\psi^{\ab, 0}"] & 
 \Jac^0_{V}(\overline{K}) 
 \end{tikzcd}
\] 
where the map $\psi^{\ab,0}$ is a group homomorphism.
 
We further know by \cite[Chapter V,Theorem 1]{SerreClass}, that the natural map from the $r_U$-fold product
$U_{\bar K}^{r_U} \rightarrow \Jac^0_{U,\bar K}$ is dominant. Since in a connected algebraic group any dense open subset generates the whole group by a sum of two elements, this implies that any element of 
$\Jac^0_{U}(\overline{K})$ is a sum of at most $2r_U$ many elements in the image of $j_{\bar u}$. It follows that the image of $\psi^{\ab,0}$ is contained in the image of the natural map from the $2r_U$-fold product $V_{\bar K}^{2r_U} \rightarrow \Jac^0_{V,\bar K}$. By assumption and the calculation above, this map is not dominant. Thus the Zariski closure $B$ of $\psi^{\ab,0}( \Jac^0_{U}(\overline{K}))$ in $ \Jac^0_{V,\bar K}$ is a proper subgroup (note that $\psi^{\ab,0}$ is a group homomorphism and thus this Zariski closure is a group).

Now the image of $\psi(\overline{K}) \colon U(\overline{K}) \rightarrow V(\overline{K})$ is contained in $j_{\bar v}^{-1}(B)$. As $j_{\bar v}(V)$ generates $\Jac^0_{V,\bar K}$ and $B$ is a proper subgroup, the intersection $B \cap j_{\bar v}(V)$ must be finite, and so $\psi(\overline{K})$ has finite image.
\end{proof} 

\begin{proposition} \label{prop:hyperbolicopenpi1}
Let  $U$ and $V$ be geometrically connected smooth curves over a Kummer-faithful field $K$ of characteristic zero,  
and assume that $\pi_1(V_{\overline{K}},\bar v)$ is non-trivial. Let  
\[
\psi \colon \pi_1(U,\bar u) \rightarrow \pi_1(V,\bar v)
\]
be a point-theoretic morphism over $\Gal_K$. 

Then either $\psi$ is open, or the image of $\psi$ is contained in a single decomposition group and a fortiori $\psi(\overline{K})$ is constant.
\end{proposition}

\begin{proof}
    Let us suppose that $\psi$ is not open:
    then 
    \[
    \psi(\pi_1(U,\bar u)) = \bigcap_{i \in I} \pi_1(V_i,\bar v_i) \subseteq \pi_1(V,\bar v)
    \]
    for pointed finite \'etale covers $V_i \rightarrow V$  with $V_i$ geometrically connected 
    the degree of $V_i \rightarrow V$ becomes arbitrarily large. Then, by the 
    Riemann--Hurwitz formula, the Euler characteristic 
    \[
    \chi_{V_i} = \deg(V_i \to V) \cdot \chi_V
    \]
    of the $V_i$ becomes an arbitrarily ``large'' negative number. 
    The induced map 
    \[
    		\psi_i \colon \pi_1(U,\bar u) \rightarrow \pi_1(V_i,\bar v_i)
    \]
    is still point-theoretic. 
    Hence, by \cref{prop:eulerchar} the image of $\psi_i(\overline{K})$ is finite, 
    and thus also the image of $\psi(\overline{K})$ is finite.
     
     By choosing $\bar u$ and $\bar v$ as in the proof of \cref{prop:eulerchar} 
     we can view $V(\overline{K})$ as embedded in $\Jac^0_V(\overline{K})$. 
     From what we just showed, the subgroup generated
     by the image of $\psi(\overline{K}) $
     \[
     \langle \im \big(\psi(\overline{K})  \colon U(\overline{K}) \rightarrow 
     V(\overline{K}) \big)\rangle \subseteq \Jac^0_V(\overline{K})     
     \]
     is finitely generated.
     However, it is also a quotient of the divisible group $\Jac^0_U(\overline{K})$. 
     Thus, it must be trivial, which in turn implies that the image of $\psi(\overline{K})$ is just one point,
     namely $\bar v$. 
     
     For any open subgroup $H \subseteq \pi_1(V,\bar v)$ the same conclusion holds for 
     the induced map $\psi^{-1}(H) \rightarrow H$ interpreted as the point-theoretic map between 
     fundamental groups of respective finite \'etale covers. 
     This implies that any decomposition group $D_{\tilde{x}} \subseteq \pi_1(U,\bar u)$ has image 
     contained in the same, fixed decomposition group $D_{\tilde{v}}$ of $\pi_1(V,\bar v)$.
     
      Let $\Delta \subseteq \pi_1(U,\bar u)$ be the (closure) of the subgroup generated by all 
      decomposition groups of $\pi_1(U,\bar u)$. We see that 
     $\psi(\Delta)$ is open in $D_{\tilde{v}}$. Furthermore, for any $\gamma \in \pi_1(U,\bar u)$, we have  
     \[
     \psi(\gamma)\psi(\Delta)\psi(\gamma)^{-1} = \psi(\Delta).
     \]
      This implies that $\psi(\gamma)$ lies in the normalizer $N_{\pi_1(V,\bar v)}(\psi(\Delta))$ of the image $\psi(\Delta)$.  But, for any  open subgroup $D  \subseteq D_{\tilde{v}}$, it is known\footnote{Mochizuki proves the required claim for local fields, but the same proof goes through for Kummer-faithful fields.} that 
      \[
      N_{\pi_1(V,\bar v)}(D) \subseteq D_{\tilde{v}},
      \]
      see \cite[Theorem 1.3]{MochizukiGalois}. Thus, $\psi(\pi_1(U,\bar u))$ is contained in $D_{\tilde{v}}$, 
      and clearly $D_{\tilde{v}}$ is not open.
\end{proof}

\begin{proposition} \label{prop:hyperbolicopen}
Let $U$ and $V$ be geometrically connected smooth curves over a Kummer-faithful field $K$.
Suppose $\ph \colon   U_{\et} \rightarrow V_{\et}$ is a pinned morphism of topoi over $K_{\et}$ such that 
the map 
\[
\ph(\overline{K}) \colon U(\overline{K})\rightarrow V(\overline{K})
\]
is non-constant. Then the induced map 
\[
\ph_\ast \colon \pi_1(U,\bar u) \rightarrow \pi_1(V,\bar v)
\]
is open.
\end{proposition}
\begin{proof}
\cref{lemma:pinned_yields_pointtheoretic}  shows that $\ph_\ast \colon \pi_1(U,\bar u) \rightarrow \pi_1(V,\bar v)$ is point-theoretic. 
We first assume that $V$ is hyperbolic. Then $j_V: V \to \Jac^1_V$ is an embedding into a torsor under the semi-abelian variety $\Jac_V$, so that \cref{prop:pinnedpoint} shows that $\ph$ and $\ph_\ast$ induce the same map on $\overline{K}$-points. Therefore 
$\ph_\ast(\overline{K})$ is also non-constant. Hence   \cref{prop:hyperbolicopenpi1} applies and  concludes the proof.

In the general case, we may reduce to $V$ being hyperbolic as follows.
For a sufficiently small Zariski open $j' \colon  V' \to V$ we consider $j \colon  U' = \ph^{-1}(V') \to U$, which is also
Zariski open. The pinned map $\ph$ of topoi induces a pinned map $\ph' \colon U'_{\et} \to V'_{\et}$ of topoi over $K_{\et}$.
There is a commutative diagram 
\[
\begin{tikzcd} 
	\pi_1(U',\bar u') \arrow[r,"\ph'_\ast"]  \arrow[d, "j'_\ast"] 
	&  \pi_1(V',\bar v') \arrow[d, "j_\ast"]  \\ 
	\pi_1(U,\bar u) \arrow[r, "\ph_\ast"] 
	& \pi_1(V,\bar v) 
\end{tikzcd}
\]
with surjective vertical maps. The upper horizontal map is open by the hyperbolic case, and so is the lower map. 
\end{proof}

\section{\'Etale reconstruction for sub-\texorpdfstring{$p$}{p}-adic fields}

The category $\Sch_K^{\ft}$ of schemes of finite type over a field $K$ localized along the class of universal homeomorphisms is denoted by 
\[
\Sch_K^{\ft}[\UH^{-1}].
\]
We now prove \cref{intro_thm:etale_reconstruction}, stated again for the convenience of the reader.
\begin{theorem}
	\label{thm:etale_reconstruction}
	Let $ K $ be a sub-$p$-adic field.
	Then the functor
	\begin{equation*}
		(-)_{\et} \colon \Schft_{K}[\UH^{-1}] \longrightarrow \RToppin_{K}
	\end{equation*}
	sending an $X/K$ of finite type to its \'etale topos $\Xet \to K_{\et}$ is fully faithful.
\end{theorem}

Recall that the inclusion of the category of seminormal schemes of finite type over $K$ into $\Sch_K^{\ft}$ has a right-adjoint, the seminormalization. In characteristic $0$ seminormalization agrees with absolute weak normalization, see \stacks{0EUS} and \stacks{0H3G}. Thus the functor seminormalization realizes the localization functor 
\[
\Sch_K^{\ft} \longrightarrow \Sch_K^{\ft}[\UH^{-1}]
\]
since $K$ has characteristic $0$, see \cite[Lemma 2.16]{CHWvoevodskyanabelian}. Seminormalization agrees with absolute weak normalization in characteristic $0$, therefore \cite[Theorem~5.15]{CHWvoevodskyanabelian} reduces the proof of  \cref{thm:etale_reconstruction} to the following statement (we omit the assumption of a $K$-rational point, because we do not need it).

\begin{proposition} \label{prop:firstreductionThmA}
Let $ K $ be a sub-$p$-adic field, and let $X$ be a smooth connected scheme of finite type over $K$. Then for any pinned morphism $\ph \colon \Xet \to \bA^1_{K,\et}$  over $K_{\et}$ there is a map $f: X \to \bA^1_K$ over $K$ such that 
\[
f (\overline{K}) \colon X(\overline{K}) \to \bA^1_K(\overline{K}) = \overline{K}
\]
agrees with the map $\ph(\overline{K})$.
\end{proposition}

We now replace $\bA^1_K$ by the hyperbolic curve $\bP^1_K \setminus \{0,1,\infty\}$ in \cref{prop:firstreductionThmA}.

\begin{lemma} \label{lemma:hyperbolicreduction}
Let $K$ be a sub-$p$-adic field.  \cref{thm:etale_reconstruction} follows from \cref{prop:ThmAinputMochizukiHom} below.
\end{lemma}
\begin{proof}
Suppose given a  pinned morphism $\ph \colon  X_{\et} \rightarrow \bA^1_{K,\et}$. 
We cover $\bA^1_K$ by two Zariski opens $U_1 = \bA^1_K \setminus \{0,1\}$ and $U_2 = \bA^1_K \setminus \{2,3\}$, both of which are isomorphic to $\bP^1_K \setminus \{0,1,\infty\}$. The preimages $X_i = \ph^\ast(U_i)$ form a Zariski open covering $X = X_1 \cup X_2$ of $X$ and $\ph$ restricts to pinned morphisms $\ph_i : X_{i,\et} \to U_{i,\et}$. By assumption,  \cref{prop:ThmAinputMochizukiHom} yields morphisms $f_i : X_i \to U_i$ over $K$ that agree with $\ph_i(\overline{K})$ on $\overline{K}$-points. It follows that $f_1$ and $f_2$ agree on $X_1 \cap X_2$ and thus glue to a map $f: X \to \bA^1_K$ with $f(\overline{K}) = \ph(\overline{K})$. This verifies the criterion for  \cref{thm:etale_reconstruction} provided by \cref{prop:firstreductionThmA}.
\end{proof}

\begin{lemma} \label{lem:hyperbolically copnnected}
Let $X$ be a smooth geometrically connected scheme of finite type over the field $K$. 
For any two points $x,y \in X(\overline{K})$ there is a geometrically connected smooth curve $C$ over $K$ and a 
map $g: C \to X$ over $K$ such that  both $x,y$ are contained in the image $g(C(\overline{K}))$.
\end{lemma}
\begin{proof}
In order to connect $x,y \in X(\overline{K})$ we consider their Galois orbits, which is $Z(\overline{K})$ for a $0$-cycle $Z \subseteq X$. 
By \cite[Corollary~1.9]{CharlesBertini} there is a geometrically irreducible curve $C' \subseteq X$ that passes through $Z$. Let $C \to C'$ be the normalization. Then $C$ is smooth and geometrically connected over $K$, and 
$x$ and $y$ are in $g(C(\overline{K}))$ for $g: C \to C' \to X$.
\end{proof}

\begin{proposition} \label{prop:ThmAinputMochizukiHom}
Let $ K $ be a sub-$p$-adic field, and let $X$ be a smooth connected scheme of finite type over $K$. Then for any pinned morphism $\ph \colon \Xet \to (\bP^1_K \setminus \{0,1,\infty\})_{\et}$  over $K_{\et}$ there is a map $f: X \to \bP^1_K \setminus \{0,1,\infty\}$ over $K$ such that 
\[
f (\overline{K}) \colon X(\overline{K}) \to \bP^1_K \setminus \{0,1,\infty\}(\overline{K})
\]
agrees with the map $\ph(\overline{K})$.
\end{proposition}

\begin{proof}
By \cref{lemma:pinned_yields_pointtheoretic} the homomorphism 
\[
\ph_\ast: \pi_1(X,\bar x) \to \pi_1(\bP^1_K \setminus \{0,1,\infty\}, \tbp) 
\]
is point-theoretic, so that $\ph_\ast(\overline{K})$ is well defined by \cref{cor:pointtheoretic_actualmap} and agrees with $\ph(\overline{K})$ by \cref{prop:pinnedpoint}. Therefore we will search for a map $f$ with $f (\overline{K}) = \ph_\ast(\overline{K})$.

If $\ph_\ast$ is open, then the celebrated theorem of Mochizuki \cite[Theorem A]{MochizukiHom} shows that there is a map $f: X \to \bP^1_K \setminus \{0,1,\infty\}$ over $K$ such that $\ph_\ast$ equals $\pi_1(f)$ up to conjugation by an element of the geometric fundamental group of $\bP^1_K \setminus \{0,1,\infty\}$. In particular, then $f(\overline{K})$ equals $\ph_\ast(\overline{K})$, and the proof is complete in this case. 

We may now assume that $\ph_\ast$ is not open, and we are going to prove that $\ph(\overline{K})$ is constant.
We first show that $\ph(\overline{K})$ is constant on the image of a map $g\colon  C\to X$ from a geometrically connected smooth curve $C$ over $K$.  
The map $\ph \circ g_{\et} \colon C_{\et} \to (\bP^1_K \setminus \{0,1,\infty\})_{\et}$ is pinned as a composition of pinned maps. The induced map  
\[
(\ph \circ g_{\et})_\ast = \ph_\ast \circ \pi_1(g)  \colon \pi_1(C,\bar c) \to \pi_1(\bP^1_K \setminus \{0,1,\infty\}, \tbp) 
\]
is not open because its image is contained in the image of $\ph_\ast$. Hence, by  \cref{prop:hyperbolicopen} applied to $\ph \circ g_{\et}$, the map 
\[
\ph(\overline{K}) \circ g_{\et}(\overline{K}) = (\ph \circ g_{\et})(\overline{K}) \colon C(\overline{K}) \to  \bP^1_K \setminus \{0,1,\infty\}(\overline{K})
\]
is constant. If $X$ is geometrically connected, then \cref{lem:hyperbolically copnnected} shows that $\ph(\overline{K})$ is constant. If $X$ is not geometrically connected, then no such map $g\colon C \to X$ exists. 

Assume first that $X$ is geometrically connected.  By what we just proved then $\ph(\overline{K})$ is constant, say with image $a \in \big(\bP^1_K \setminus \{0,1,\infty\}\big)(\overline{K})$. Pick any of the maps $g\colon C \to X$ used in the previous paragraph. Then 
$\ph_\ast \circ \pi_1(g)$ factors over one of the decomposition groups $D_a$ in the conjugacy class of decomposition subgroups associated to the point $a$. As $C$ is geometrically connected over $K$ the composition 
\[
\pi_1(C,\bar c) \xrightarrow{\ph_\ast \circ \pi_1(g)} D_a \subseteq \pi_1(\bP^1_K \setminus \{0,1,\infty\}, \tbp) \xrightarrow{\pr_\ast} \Gal_K
\]
is surjective. This implies that $a \in X(K)$ is a $K$-rational point. Hence there is a map 
\[
f\colon X \longrightarrow \Spec(K) = \Spec(k(a)) \longrightarrow \bP^1_K \setminus \{0,1,\infty\},
\]
which agrees with $\ph(\overline{K})$ on $\overline{K}$-points. Thus, we have proven our statement when $X$ is geometrically connected. 

On the other hand, if $X$ is not geometrically connected, then by a base change to some finite Galois extension $L/K$ contained in $\overline{K}$, we can assume that all connected components of the base change $X_L$ have an $L$-rational point and therefore are geometrically connected over $L$. Hence, by performing the above argument to each connected component, we can construct a map 
\[
f_L \colon  X_L \longrightarrow \bP^1_L \setminus \{ 0,1,\infty \}
\]
such that  $\ph_L(\bar{K}) = f_L(\bar{K})$, where $\ph_L$ is the base change $(X_L)_{\et} \to (\bP^1_L \setminus \{0,1,\infty\})_{\et}$ of $\ph$. Now $\ph_L$ is invariant under $\Gal(L/K)$, so also $f_L(\overline{K})$ is Galois invariant and thus, moreover, the map $f_L$ is Galois invariant. By Galois descent, the map $f_L$ descends to 
a map $f  \colon  X \rightarrow \bP^1_K \setminus \{0,1,\infty\}$ such that $\ph(\bar{K}) = f(\bar{K})$.

This completes the proof of  \cref{thm:etale_reconstruction}.
\end{proof}

\section{A topos-theoretic section conjecture for normal base schemes}

In this section we give a positive answer to a question of Voevodsky. Recall that the function field of an irreducible scheme is the residue field at its generic point.

\begin{proposition}[{\cite[Conjecture, p. ~515]{MR1098621}}]  \label{prop:voevodskysection}
Let $S$ be a normal noetherian scheme, and let $X \to S$ be a morphism of finite type. Assume that the function field of each connected component of $S$ is a perfect field. Let $U \to S$ be a quasi-compact separated \'etale morphism. Then every chivalrous map $\sigma \colon U_{\et} \to X_{\et}$ over $S_{\et}$ is induced by a unique morphism of schemes $s \colon U \to X$ over $S$.
\end{proposition}

We first define the following class of geometric morphisms. 

\begin{definition} \label{def:fiberwisepinned}
Let $S$ be a qcqs scheme and let $X \to S$ and $Y \to S$ be qcqs schemes over $S$. A geometric morphism 
$\varphi \colon X_{\et} \to Y_{\et}$  
over $S$ is \defemph{fibrewise pinned} if, for every $s \in S$, the induced map 
\[
	|\varphi|_s  \colon |X_s| \to |Y_s|
\]
sends closed points to closed points. The full subgroupoid spanned by fiberwise pinned geometric morphisms is denoted by
\[
\Hom^{\fpin}_S(X_{\et},Y_{\et}) \subseteq \Hom_S(X_{\et},Y_{\et}).
\]
\end{definition}

Note that if $S$ is qcqs and $X$ and $Y$ are qcqs and of finite type over $S$, then
\[ 
	\Hom^{\che}_S(X_{\et},Y_{\et}) \subseteq \Hom^{\fpin}_S(X_{\et},Y_{\et}), 
\]
see \cref{rmk:basechangechivalrous}. Further, $\Hom^{\fpin}_S(X_{\et},Y_{\et})$ is then equivalent to a set by the proof of \cite[Proposition~3.10]{CHWvoevodskyanabelian}. Lastly, the natural map $\Hom_S(X,Y) \to \Hom^{\fpin}_S(X_{\et},Y_{\et})$ factors through $\Hom^{\che}_S(X_{\et},Y_{\et})$ by Chevalley's theorem when $X$ and $Y$ are further of finite presentation over $S$. In the situation when $\Hom^{\fpin}_S(X_{\et},Y_{\et})$ or $\Hom^{\che}_S(X_{\et},Y_{\et})$ is equivalent to a set, we will make that identification without further comments.

\begin{remark} \label{rmk:restrictionfiberwisepinned}
Let $S$ be a qcqs scheme and let $X$ and $Y$ be qcqs schemes of finite type over $S$. Let $\varphi \colon X_{\et} \to Y_{\et}$ be a \emph{coherent} chivalrous geometric morphism over $S$. Given $V \to Y$, a separated quasi-compact \'etale map, consider the induced map $\varphi^*(V)_{\et} \to V_{\et}$ of topoi over $S_{\et}$ and note that $\varphi^*(V)$ is represented by a quasi-compact separated \'etale scheme by \cite[Theorem A.8]{CHWvoevodskyanabelian}. Then, as pointed out in \cref{rmk:chivalrousrestriction}, it is not immediate that $\varphi^*(V)_{\et} \to V_{\et}$ is chivalrous. However, if  $\varphi \colon X_{\et} \to Y_{\et}$ is a fiberwise pinned
geometric morphism over $S$, it is immediate that $\varphi^*(V)_{\et} \to V_{\et}$ is fiberwise pinned. This follows essentially since any quasi-compact separated \'etale map is quasi-finite.
\end{remark}

The following injectivity result is a strengthening of \cite[Theorem~4.8]{CHWvoevodskyanabelian}.

\begin{lemma} \label{injectivitystrengthening}
Let $S$ be a qcqs scheme.  Let $X$ and $Y$ be schemes of finite presentation over $S$. If $X$ is reduced, then the natural map 
\[ 
	\Hom_S(X,Y) \to \Hom^{\che}_S(X_{\et},Y_{\et})  \subseteq \Hom^{\fpin}_S(X_{\et},Y_{\et})  
\] 
is injective.
\end{lemma}
\begin{proof}
The proof is basically the same as that of \cite[Theorem~4.8]{CHWvoevodskyanabelian}. 
Let $X_{\fcl}$ be the set of points $x \in X$ which are closed in their fibers: $x$ is closed in $X_{s}$, where $s$ is the image of $x$ under the map $X \to S$. Since $X$ is reduced and every fiber is a Jacobson scheme, the natural map
\[
\coprod_{x \in X_{\fcl}} \Spec k(x) \to X
\]
is an epimorphism. This implies that the leftmost vertical map is injective in the commutative diagram
\[
	\begin{tikzcd}
		\Hom_S(X,Y) \arrow[r] \arrow[d,hook] 
		& \Hom^{\fpin}_S(X_{\et},Y_{\et}) \arrow[d] 
		\\
		\prod_{x \in X_{\fcl}} \Hom_S(\Spec k(x) ,Y) \arrow[r] 
		& \prod_{x \in X_{\fcl}}  \Hom^{\fpin}_S((\Spec k(x))_{\et},Y_{\et}).
	\end{tikzcd}
\]
To prove the lemma, it therefore suffices to prove that
\[
	\Hom_{S}(\Spec k(x),Y)
	\to
	\Hom^{\fpin}_{S}((\Spec k(x))_{\et},Y_{\et})
\]
is injective for every $x \in X_{\fcl}$. Since $\Hom^{\fpin}_S((\Spec k(x))_{\et},Y_{\et}) = \Hom^{\fpin}_{k(s)}((\Spec k(x))_{\et},(Y_s)_{\et})$ by \cite[Lemma~3.6]{CHWvoevodskyanabelian}, it is enough to show that 
\[ 
	\Hom_S(\Spec k(x),Y) = \Hom_{k(s)}(\Spec k(x),Y_s) \to \Hom_{k(s)}((\Spec k(x))_{\et},(Y_s)_{\et}) 
\]  
is injective. This follows from \cite[Theorem~4.8]{CHWvoevodskyanabelian}.
\end{proof}

The following well-known technical lemma will be used repeatedly throughout this section. 

\begin{lemma} \label{lem:factorisation}
 Let $S$ and $X$ be schemes and let $\sigma \colon S_{\et} \to X_{\et}$ be a geometric morphism.
\begin{enumerate}[label=(\arabic*), ref=(\arabic*)]
	\item \label{lemitem:factorization-open}
	Let $U \subseteq X$ be a Zariski open subscheme. Then the following are equivalent.
		\begin{enumerate}[label=(\alph*), ref=(1\alph*)]
			\item  \label{lemitem:factorization-open-1}
			The geometric morphism $\sigma$ factors over $U_{\et}$.
			\item \label{lemitem:factorization-open-2}
			$|\sigma| \colon |S| \to |X|$ factors through $|U|$.
			\item \label{lemitem:factorization-open-3}
			$\sigma^\ast(U)$ is isomorphic to the terminal object of $S_{\et}$.
		\end{enumerate}
	\item \label{lemitem:factorization-closed}
	Let $Z \inj X$ be a closed subscheme, and let $W \subseteq X$ be the Zariski open complement of $Z$.
	Then the following are equivalent.
		\begin{enumerate}[(a)]
			\item
			The geometric morphism $\sigma$ factors over $Z_{\et}$.
			\item
			$|\sigma| \colon |S| \to |X|$ factors through $|Z|$.
			\item
			$\sigma^\ast(W) = \emptyset$.
		\end{enumerate}	
\end{enumerate}
\end{lemma}
\begin{proof}
We note that \cite[Recollection 3.4, Example 3.5]{CHWvoevodskyanabelian} imply that
\labelcref{lemitem:factorization-open-2} 
and \labelcref{lemitem:factorization-open-3} 
are equivalent. There is a canonical equivalence $U_{\et} \cong (X_{\et})_{/U}$, where the latter denotes the slice topos over the \'etale sheaf represented by $U$ \stacks{04IQ}. By \stacks{04H1}, if $\sigma^*(U)$ is isomorphic to the terminal object of $S_{\et}$, then $\sigma$ factors through $U_{\et}$, and conversely. 

The equivalence of the conditions in \labelcref{lemitem:factorization-closed} follows immediately from \cite[Recollection~4.3, Remark~4.4]{CHWvoevodskyanabelian}.
\end{proof}

We now actually prove the following slightly stronger variant of \cref{prop:voevodskysection}. For simplicity we assume that $S$ is connected. The general case of Voevodsky's conjecture follows immediately. 

\begin{proposition} \label{prop:quasifinitesection}
Let $X \to S$ be a morphism of finite type over a connected normal noetherian scheme $S$ with perfect function field $K$. Let $q \colon W \to S$ be a dominant quasi-finite morphism, 
and assume that $W$ is connected and normal. Then the natural maps
\[ 
	\Hom_S(W,X) \isomto \Hom_S^{\che}(W_{\et},X_{\et})  \isomto   \Hom_S^{\fpin}(W_{\et},X_{\et}) 
\] 
are bijective. 
\end{proposition}

\begin{proof}
\textit{Step 1:} 
Since the second map is the inclusion of a subset, and since the composition is injective by  \cref{injectivitystrengthening}, it suffices to show that every $\sigma \in \Hom^{\fpin}_S(W_{\et},X_{\et})$ is induced by an $S$-morphism $W \to X$.

\smallskip

\textit{Step 2:} 
Assume that $\sigma = f_{\et}$ for some $S$-map $f \colon W \to X$. Let $U \subseteq X$ be Zariski open and $V = \sigma^\ast(U)$ the Zariski open of $W$ obtained from pullback. It follows that $f^{-1}(U) = V$, and the restriction $\sigma|_V \colon V_{\et} \to U_{\et}$ agrees with the geometric morphism $f|_{V,\et}$, which is induced by $f|_V \colon V \to U$. 

\smallskip

\textit{Step 3:} 
Choose an open covering $S = \cup_{i \in I} S_i$ and for each $i$ an open covering $X_i := X_{S_i} = \cup_{j \in {J_i}} X_{ij}$ of the Zariski-open $X_i \subseteq X$. Set $W_{ij} = \sigma^\ast (X_{ij})$ and $W_i = \sigma^\ast(X_i)$, and note that $W_{ij}$ and $W_i$ are Zariski opens of $W$, and that $\cup W_{ij} = W$. If each 
\[ 
	\sigma_{ij} = \sigma|_{X_{ij}} \colon W_{ij,\et} \to X_{ij,\et} 
\] 
is induced by an $S$-map $f_{ij} \colon W_{ij} \to X_{ij}$, then, by steps 1\&2 and \cref{injectivitystrengthening},  the $f_{ij}$ agree on overlaps. This implies that they glue to a unique $S$-morphism $f \colon W \to X$. 

Moreover, we claim that  $\sigma \cong f_\et$. Indeed, for every \'etale sheaf $\mathcal{F}$ on $X$, the isomorphisms $\sigma_{ij} \cong f_{ij,\et}$ show that we have isomorphisms $(\sigma^* \mathcal{F})_{W_{ij}} \cong (f^*\mathcal{F})_{W_{ij}}$ on $W_{ij,\et}$. These isomorphisms agree on overlaps since fiberwise pinned geometric morphisms have no nontrivial automorphisms, hence they glue to a natural isomorphism $\sigma^* \cong f^*$. This implies that $\sigma \cong f_{\et}$. 

From now on we can assume that $X = \Spec(A)$ and $S = \Spec(B)$ are affine. 

\smallskip

\textit{Step 4:} 
Since $W \to S$ is dominant and quasi-finite, its generic fiber is $W_K = \Spec(L)$ for some finite extension $L/K$. The restriction of $\sigma$ to the generic fiber is pinned by assumption, thus 
\cite[Lemma~3.6 and Proposition~4.5]{CHWvoevodskyanabelian} 
show that there exists a unique map $t \colon \Spec(L)  \rightarrow X_K$ such that $t_{\et} \cong \sigma_K$ as $K$-maps $L_\et \to X_{K,\et}$.

\smallskip

\textit{Step 5:} 
Since $W \to S$ is quasi-finite, the base-change map 
\[ 
	(X \times_S W)_\et \isomto X_{\et} \times_{S_{\et}} W_{\et}
\] 
is an equivalence in $\RTop$ by \cite[Proposition 6.21]{haine2025condensed}. 
By the universal property, the map $\sigma$ lifts to a section $\vartheta \colon W_\et \to (X \times_S W)_\et$ over $W_{\et}$. The $L$-point $t$ of $X$ lifts to an $L$-point also denoted by $t$ in 
\[ 
	X \times_S W(L), 
\] 
whose image has residue field $L$. We denote by $Z \inj X \times_S W$ the Zariski closure of the image. Since $\vartheta$ is continuous, the map $|\vartheta|$ factors through $|Z| \subseteq |X \times_S W|$. Hence, by \cref{lem:factorisation}, $\vartheta$ factors through $Z_{\et}$. We also denote this geometric morphism by $\vartheta \colon W_\et \to Z_\et$.

\smallskip

\textit{Step 6:} 
Let $w \in W$ be a point of codimension one and let $z \in Z$ be the image $|\vartheta|(w)$. Since $\vartheta$ is a section, $z$ projects onto $w$ and the pullback along $Z \to W$ induces a local map $\dO_{W,w} \to \dO_{Z,z}$. Since $t$ identifies their fraction fields with $L$, and since valuation rings are maximal local subrings of their fraction fields, the map $\dO_{W,w} \isomto \dO_{Z,z}$ is an isomorphism. Thus $t \colon \Spec(L) \to X$ extends uniquely as an $S$-map to a neighbourhood of $w \in W$. 

Let $j \colon U \hookrightarrow W$ be the maximal Zariski-open to which $t$ extends as an $S$-map $U \to X$. From what we have just proven, $W \setminus U$ has codimension at least $2$ in $W$. Hence, $\dO_W \to j_\ast\dO_U$ is an isomorphism as $W$ is normal. Because $X = \Spec(A)$ and $S= \Spec(B)$ are affine, we obtain
\[
	\Hom_S(W,X) = \Hom_B \big(A,\dO(W)\big) = \Hom_B \big(A,\dO(U)\big) = \Hom_S(U,X).
\]
Therefore the morphism $U \to X$ extends to a unique $S$-map $f \colon  W \to X$.

\smallskip

\textit{Step 7:} 
It remains to see that $f_{\et}$ agrees with $\sigma$. Both the map $f$ and $X \to S$ are separated because $X$ is affine. Further, the closed subscheme $Z \inj X \times_S W$ identifies with the graph of $f$. Therefore both $f \colon W \to Z$ and $\vartheta \colon W_\et \to Z_\et$ induce sections of the map $Z_\et \to W_\et$ and thus must be isomorphic, since the projection $Z \isomto W$ is an isomorphism.
\end{proof}
\begin{corollary} \label{cor:quasifinitesection-general}
Let $X \to S$ be a morphism of finite type over a noetherian scheme $S$. Let $q \colon W \to S$ be a quasi-finite morphism, with $W$ normal, and assume that the function field of every connected component of $W$ is perfect. Then the natural maps
\[ 
	\Hom_S(W,X) \isomto \ \Hom_S^{\che}(W_{\et},X_{\et})  \isomto  \Hom_S^{\fpin}(W_{\et},X_{\et}) 
\] 
are bijective. 
\end{corollary}
\begin{proof}
Injectivity follows from \cref{injectivitystrengthening}. For surjectivity, let $\sigma \colon W_{\et} \to X_{\et}$ be fiberwise pinned. By \cite[Proposition 6.21]{haine2025condensed}, we obtain a section $\vartheta \colon W_{\et} \to (X \times_S W)_{\et}$ over $W_{\et}$. Note that $\vartheta$ is fiberwise pinned since $W \to S$ is quasi-finite. Indeed, if $w \in W$, and $s \in S$ is its image, $w$ is closed in $W_s$. Since $\sigma$ is fiberwise pinned, $x = |\sigma|(w)$ is closed in $X_s$. This shows that $\Spec k(x) \to \Spec k(s)$ is finite, and this implies that the inverse image of $x$ under the projection 
\[ 
	(X \times_S W)_w \to X_s 
\] 
is finite over $\Spec k(w)$. Thus $|\vartheta|(w)$ is closed in $(X \times_S W)_w$. \cref{prop:quasifinitesection} applied to $\vartheta$ and each connected component of $W$ separately provides a map $t \colon W \to X \times_S W$ such that $t_{\et} \cong \vartheta$. Composing $t$ with the projection defines a map $f \colon W \to X \times_S W \rightarrow X$ that satisfies $f_{\et} \cong \sigma$.
\end{proof}

\section{\'Etale reconstruction over general base schemes}

In this section we will prove \cref{intro_thm:etalereconstructiongeneral}. The strategy will use \cref{intro_thm:etale_reconstructionB} to obtain a map in the generic fibre and then use normal quasi-sections (see \cref{def:quasisections}) and \cref{prop:quasifinitesection} to spread out the map and to control its effect on \'etale topoi. For that reason we need to impose topological conditions on the base $S$ and the map $X \to S$. Namely, we need $S$ to be Nagata so that normalization remains of finite type over $S$, and we need $S$ geometrically unibranch and $X \to S$ to be componentwise universally open (see \cref{def:cuop}) so that the normalization $X^\nu$ admits enough normal quasi-sections over $S$, see \cref{prop:normalization-has-enough-qs}.

\subsection{Componentwise universally open morphisms}

Following Grothendieck \cite{MR1483108}, one aspect of anabelian geometry for schemes flat and of finite type over $\Spec(\bZ)$ studies the functor \emph{\'etale topos over~$\bZ_\et$}
\[
	(-)_{\et} \colon \Schft_{\bZ}[\UH^{-1}] \longrightarrow \RTopchi_{\bZ}
\]
defined on the localisation by the class of universal homeomorphisms. Our approach over a more general base scheme $S$ (not just $\bZ$) depends on the presence of enough normal quasi-sections 
whose generic points have characteristic $0$. Therefore we need to impose an extra topological condition (that for schemes of finite type over $\bZ$ is implied by flatness) that we define next.

Recall that any scheme $X$ is the union of its irreducible components: the Zariski closures of points minimal with respect to the partial order of specialisation.  The function field of the irreducible component is the residue field of its generic point. 

\begin{definition} \label{def:cuop}
A morphism $f \colon X \to S$ is  \defemph{componentwise universally open} if for every irreducible component $Z$ of $X$  the restriction $f|_Z  \colon Z \to S$ is a universally open morphism: it is open and remains an open morphism after arbitrary base change. We denote by 
\[
	\Schfft_S
\]
the category of morphisms $X \to S$ of finite type that are componentwise universally open.
\end{definition}

\begin{remarks} \label{rmk:classesofunivopen}
\begin{enumerate}[(1)]
	\item
	Suppose $S$ is a Dedekind scheme. A morphism $X \to S$ of finite type is componentwise universally open exactly when each irreducible component $Z$ of $X$ dominates a connected component of $S$. Indeed, then $Z \to S$ is flat of finite type and therefore universally open. In particular, any flat scheme of finite type over $S$ is componentwise universally open. 

	\item \label{rmkitem:cuop-is-uop}
	A componentwise universally open morphism $f \colon X \to S$ is universally open. Indeed, for any $S' \to S$, the base change $X \times_S S'$ is covered by $X_\alpha \times_S S'$, where $X_\alpha$ ranges over the irreducible components of $X$. Then for every open $U \subseteq X \times_S S'$ its image in $S'$ is covered by the images of $U \cap (X_\alpha \times_S S')$, which are open by assumption. 
	
	\item \label{rmkitem:normal-case}
	Suppose $X \to S$ is of finite type with $X$ normal, and that $S$ is noetherian. Then $X \to S$ is componentwise universally open if and only if it is universally open.

	\item \label{rmkitem:counterexample-to-uop-implies-cuop}
	Let $k$ be a field and let $S$ be the affine plane $\bA^2_k$. Let $\BL_0(S) \to S$ be the blow-up of $S$ at the origin, and let $X = \BL_0(S) \cup_E \bP^1_S$ be $\BL_0(S)$ glued transversally along the exceptional fibre $E \simeq \bP^1_k$ with the fibre over the origin of the relative projective line $\bP^1_S = S \times_k \bP^1_k$. Then $X \to S$ is universally open, but not componentwise universally open.
\end{enumerate}
\end{remarks}

We now define the notion of a quasi-section. Such quasi-sections will be crucial for the proof of \cref{thm:etalereconstructiongeneral}, allowing us to propagate the reconstruction result from the generic fibre to schemes componentwise universally open over $S$.  

\begin{definition} \label{def:quasisections}
A \defemph{quasi-section} of a morphism $f \colon X \to S$ is an $S$-morphism $q \colon W \to X$ such that $W \to S$ is quasi-finite and maps generic points of irreducible components of $W$ to generic points of irreducible components of $S$. The quasi-section is a \defemph{normal quasi-section} if $W$ is normal. 

We say that $X \to S$ has \defemph{enough (normal) quasi-sections} if, for every $s \in S$ and every closed point $x \in X_s$ of the fibre, there exists a (normal) quasi-section $q \colon W \to X$ such that $x \in q(W)$. 
\end{definition}

\begin{proposition} \label{prop:enough-normal-quasi-sections}
Every universally open morphism $f \colon X \to S$ of finite type over a noetherian Nagata scheme $S$ has enough normal quasi-sections. 
\end{proposition}
\begin{proof}
The normalization map $S^\nu \to S$ is finite surjective. It therefore suffices to show the claim for the base change $X_1 = X \times_S S^\nu \to S^\nu$. Composition of enough quasi-sections for $X_1 \to S^\nu$ with the projection $X_1 \to X$ produces enough quasi-sections for $X \to S$. Further, by restricting to a connected component of $S^\nu$ we may assume that $S$ is irreducible and geometrically unibranch. 

Let $x \in X$  be a point that is closed in its fibre $f^{-1}(f(x))$. 
Then \cite[Proposition 14.5.9]{GrothendieckEGAIV3} provides an irreducible  locally closed $W_0 \inj X$ containing $x$ such that the restriction $f|_{W_0} \colon W_0 \to S$ is quasi-finite and dominant. Let $W = W_0^\nu$ be the normalization of $W_0$ and $q \colon W \to W_0 \to X$ the composition. Then $W \to S$ is quasi-finite by the assumption that $S$ is Nagata, and so $q \colon W \to X$ is a normal quasi-section such that $x$ is contained in $q(W)$.
\end{proof}

By \cref{rmk:classesofunivopen}~\labelcref{rmkitem:cuop-is-uop} the claim of \cref{prop:enough-normal-quasi-sections} applies to componentwise universally open morphisms $X \to S$. This is not enough for our application, because we need to know the same for the normalization $X^\nu \to S$. For this transition we need to recall the notion of an equidimensional morphism.

\begin{recollection}
A morphism  $f \colon X \to S$  of finite type between noetherian schemes is \defemph{equidimensional} if the local fibrewise dimension function $x \mapsto \dim_x(f^{-1}(f(x)))$ is locally constant on $X$ and each irreducible component of $X$ dominates an irreducible component of $S$. 
\end{recollection}

The following is a variant of \cite[Proposition~14.5.9]{GrothendieckEGAIV3}.

\begin{proposition} \label{prop:equivalences-of-topological-notions}
Let $f \colon X \to S$ be of finite type with $S$ a noetherian Nagata scheme that is geometrically unibranch.  Consider the following conditions.
\begin{enumerate}[(a)]
	\item \label{propitem:equidim}
	$f$ is equidimensional.
	
	\item \label{propitem:uop}
	$f$ is universally open.
	
	\item \label{propitem:enough-qs}
	$f$ admits enough normal quasi-sections.
%
\end{enumerate}
Then the following holds:
\[
\labelcref{propitem:equidim} \Longrightarrow \labelcref{propitem:uop} \Longleftrightarrow \labelcref{propitem:enough-qs}. 
\]
If, moreover, $X$ is irreducible, then all conditions are equivalent. 
\end{proposition}
\begin{proof}
The claim \labelcref{propitem:equidim} $\Longrightarrow$ \labelcref{propitem:uop} is Chevalley's criterion \cite[Corollaire~14.4.4]{GrothendieckEGAIV3}, and the implication 
\labelcref{propitem:uop} $\Longrightarrow$ \labelcref{propitem:enough-qs} is a special case of \cref{prop:enough-normal-quasi-sections}.

Next we show \labelcref{propitem:enough-qs} $ \Longrightarrow$ \labelcref{propitem:uop}. Let $\psi \colon S' \to S$ be an arbitrary morphism and $f' \colon X' \to S'$ the base change of $f$. Let $U \subseteq X'$ be an open, and let $s' \in f'(U)$ be an arbitrary point in the image. The fibre $U \cap f'^{-1}(s')$ is an open in the fibre $X'_{s'} = f'^{-1}(s')$ which is a scalar extension of the fibre $X_s = f^{-1}(s)$ for the image $s = \psi(s')$. It follows that there is a closed point $x' \in U \cap f'^{-1}(s')$ that maps to a closed point $x \in X_s$. We now choose a normal quasi-section $g \colon W \to X$ with $x \in g(W)$ and projection $h \colon W \to S$. Let $g' \colon W' \to X'$ be the base change so that $x' \in g'(W')$, and call the base change of the projection $h' \colon W' \to S'$. 
Because $h$ is equidimensional (of relative dimension $0$), the implication 
\labelcref{propitem:equidim} $\Rightarrow$ \labelcref{propitem:uop} shows that $h$ is universally open. Therefore the set $h'(g'^{-1}(U)) \subseteq S'$ is open. Since $s' \in h'(g'^{-1}(U)) \subseteq f'(U)$, we have shown that $f'(U)$ is open in $S'$.

If $X$ is irreducible, then being equidimensional is equivalent to condition (c) of \cite[Proposition~14.5.9]{GrothendieckEGAIV3} which shows that it is equivalent to $f$ being universally open. 
\end{proof}

The example of \cref{rmk:classesofunivopen}~\labelcref{rmkitem:counterexample-to-uop-implies-cuop}  shows that in \cref{prop:equivalences-of-topological-notions} the implication \labelcref{propitem:uop} $\Rightarrow$ \labelcref{propitem:equidim} can fail when $X$ is not irreducible. 

\begin{lemma} \label{lem:normalization-equidim}
Let $S$ be a noetherian Nagata scheme, and let $f \colon X \to S$ be of finite type and equidimensional. Then the normalization $X^\nu$ of $X$ is of finite type and equidimensional over $S$. 
\end{lemma}
\begin{proof}
Since $S$ is Nagata, the normalization $X^\nu \to X$ is finite and the composition $f_\nu \colon X^\nu \to S$ is of finite type. 

We may now restrict $X$ to the connected component where the fibrewise dimension function $\dim_x(f^{-1}(f(x)))$ is constant, say equal to $d$. The normalization of the generic fibres is the normalization of an equidimensional variety of dimension $d$ and thus is also equidimensional of dimension $d$. This is the case of a point $S = \Spec(k)$ which for example follows from the dimension of an irreducible variety being equal to the transcendence degree of its function field (which does not change under normalization).

By Chevalley's theorem \cite[Th\'eor\`eme~13.1.3]{GrothendieckEGAIV3}, the dimension function $\dim_x(f_\nu^{-1}(f_\nu(x)))$ is upper semicontinuous, so all irreducible components of all fibres are of dimension at least $d$. Let $Y$ be such an irreducible component of the fibre $(X^\nu)_s = f_\nu^{-1}(s)$. Then since finite maps do not increase the dimension we have the estimates
\[
	d \leq \dim(Y) \leq \dim(X^\nu)_s \leq \dim X_s = d.
\]
It follows that $f_\nu \colon X^\nu \to S$ is equidimensional of relative dimension $d$ as well. 
\end{proof}

We can finally prove the existence of enough quasi-sections for the normalization of componentwise universally open $S$-schemes. 

\begin{proposition} \label{prop:normalization-has-enough-qs}
Let $S$ be a noetherian Nagata scheme that is geometrically unibranch. 
Let $f \colon X \to S$ be a morphism of finite type that is componentwise universally open. Let $U \to X^\nu$ be a quasi-compact \'etale open of the normalization $X^\nu$ of $X$. Then the composition $U \to S$ is universally open, admits enough normal quasi-sections, and $U \in  \Schfft_S$.

In particular $X^\nu \to S$ admits enough normal quasi-sections and $X^\nu \in  \Schfft_S$.
\end{proposition}
\begin{proof}
We may replace $X$ by an irreducible component and thus assume that $X$ is irreducible and $X \to S$ is universally open. By \cref{prop:equivalences-of-topological-notions} the map $X \to S$ is equidimensional. Hence by \cref{lem:normalization-equidim} the normalization $X^\nu$ is also equidimensional over $S$. We conclude from \cref{prop:equivalences-of-topological-notions} again that $X^\nu \to S$ is universally open.  Since \'etale maps are also universally open, the same holds for $U \to S$. By \cref{rmk:classesofunivopen}~\labelcref{rmkitem:normal-case} this implies  $U  \in  \Schfft_S$.
The claim about enough normal quasi-sections follows from \cref{prop:enough-normal-quasi-sections}.
\end{proof}
\subsection{Statement of the relative anabelian theorem}

The localisation $\Schfft_{S}[\UH^{-1}]$ can be obtained via a localisation functor to the full subcategory of absolute weakly normal $S$-schemes.

\begin{recollections}\label{rmk:absolutelyweaklynormal}
\begin{enumerate}[(1)]
\item \label{recollitem:awn-as-initial-uh}
Let $X$ be a scheme. Then the category whose objects are universal homeomorphisms $Y \to X$ has an initial object
 \[  
 	X^{\awn} \to X, 
\] 
the \defemph{absolute weak normalization} of $X$, see \stacks{0EUS}. The assignment taking $X$ to its absolute weak normalization $X^{\awn}$ assembles into a functor, which we call the absolute weak normalization functor. Absolute weak normalization is further local for the Zariski topology.

\item \label{rmkitem:awn-is-awn}
The absolute weak normalization $X^{\awn}$ of a scheme $X$ is \defemph{absolutely weakly normal}, first defined by Rydh \cite[Appendix B]{RydhSubmersions}. We recall the definition although we will mainly use the formal property stated in \labelcref{recollitem:awn-as-initial-uh}. An affine scheme $\Spec A$ is absolutely weakly normal if $A$ is seminormal \stacks{0EUL} and, for every prime $p$ and all $x,y \in A$ such that 
$p^px=y^p$, there exists a unique $a \in A$ such that 
\[ 
	x = a^p \qquad \text{and} \qquad y = pa. 
\] 
An arbitrary scheme $X$ is absolutely weakly normal if it is covered by affine opens which are absolutely weakly normal.  Since seminormal schemes are reduced \stacks{0EUQ}, any absolutely weakly normal scheme is reduced.
 \end{enumerate}
\end{recollections}

\begin{lemma} \label{lem:normalabsnorm}
Let $S$ be a noetherian Nagata scheme such that the function field of every irreducible component has characteristic zero. Let $X$ be a scheme in $\Schfft_{S}$.
\begin{enumerate}
	\item 
	The absolute weak normalization $X^{\awn} \to X$ is a finite universal homeomorphism. 
	In particular, $X^{\awn}$ is a scheme  in  $\Schfft_{S}$. 
	\item 
	Suppose $X$ is further normal. Then $X$ is absolutely weakly normal.
\end{enumerate}
\end{lemma}

\begin{proof}
We may assume that $X$ is reduced. Since $X$ is componentwise universally open over $S$, the generic point of each irreducible component has characteristic zero. Further, the map $X^{\awn} \to X$ is by assumption a universal homeomorphism, hence the associated residue field extensions on points are purely inseparable. Since the generic points of $X$ are of characteristic zero, the residue field extensions on generic points are trivial, hence $X^{\awn} \to X$ is birational. By the universal property of normalization \stacks{035Q}, the normalization $X^{\nu} \to X$ factors as $X^{\nu} \to X^{\awn} \to X$. Since $S$ is Nagata, $X^{\nu} \to X$ is finite, and since $X$ is noetherian, $X^{\awn} \to X$ is also finite. This implies that $X^{\awn}$ is of finite type over $S$, and being componentwise universally open is inherited from $X$ because $X^\awn \to X$ is a universal homeomorphism.

For the second part, note that by (1), the map $X^{\awn} \to X$ is finite. Since $X$ is normal and componentwise universally open, the generic points of its irreducible components have characteristic zero. Since $X^{\awn} \to X$ is universally injective, the residue field extensions obtained at the generic points are purely inseparable, hence trivial. This shows that $X^{\awn} \to X$ is finite and birational, hence an isomorphism since $X$ is normal.  \end{proof}

\begin{proposition} \label{prop:localisation-via-awn}
Let $S$ be a noetherian Nagata scheme such that the function field of every irreducible component has characteristic zero. Then absolute weak normalization yields a localizing equivalence
\[
	(-)^{\awn} \colon \Schfft_{S}[\UH^{-1}] \isomto \Schuniawn_S 
\]
of the localisation by universal homeomorphisms to the category of componentwise universally open $X \to S$ such that $X$ is absolutely weakly normal.
\end{proposition}
\begin{proof}
We proved in \cref{lem:normalabsnorm} that $(-)^\awn$ induces a functor
\[
	\Schfft_{S} \longrightarrow \Schuniawn_S .
\]
This is a right adjoint to the inclusion and thus localizing for those $S$-morphisms $h: X \to S$ such that $h^\awn$ is an isomorphism. These are exactly the universal homeomorphisms as explained in 
\cite[Theorem~2.7.3]{CHWvoevodskyanabelian}. 
\end{proof}

\cref{prop:localisation-via-awn} (and \cref{lem:normalabsnorm} for the claim that normal implies absolutely weakly normal) reduces the proof of \cref{intro_thm:etalereconstructiongeneral} and \cref{intro_cor:isomhigher} to the following. 
 
\begin{theorem} \label{thm:etalereconstructiongeneral}
Let $S$ be a noetherian Nagata scheme that is geometrically unibranch and such that the function field of every irreducible component is a sub-$p$-adic field. 
Let $X,Y$ be schemes in $\Schfft_{S}$. Then the \'etale topos functor induces a bijection
\[
	\Hom_S(X^{\awn},Y) \isomto  \Hom^{\che}_{S}(X^{\awn}_{\et},Y_{\et}).
\]
\end{theorem}

\begin{remark}
In \cref{thm:etalereconstructiongeneral}  it is necessary to work with chivalrous morphisms because there are 
maps $f \colon X \rightarrow Y$ of $S$-schemes as in the theorem for which $f_\et$
is not pinned. More concretely, if $Y=S = \Spec(R)$ is the spectrum of a discrete valuation ring, the projection $\pr \colon X=\mathbb{A}^1_{S}  \to S$ sends the closed point defined by the ideal $(aT-1)$, with a standard coordinate $T$ and an element $a \in R$ of positive valuation, to the generic point of $S$. Hence the induced $\pr_{\et} \colon \bA_{S,\et}^1 \to S_\et$ is not pinned. However, any map of finite-type schemes sends constructible sets to constructible sets by Chevalley's Theorem \stacks{054K}. 
\end{remark}

We will prove \cref{thm:etalereconstructiongeneral} in \cref{sec:proof-of-relative-thm} after having reduced the claim to the case of normal schemes $X$ by the following lemma.

\begin{lemma} \label{lem:h-descent}
For the proof of \cref{thm:etalereconstructiongeneral} we may suppose that $X$ is normal. 
\end{lemma}
\begin{proof}
Let $\nu \colon X^\nu \to X$ be the normalization map which factors as $\smash{X^\nu \xrightarrow{\omega} X^\awn \to X}$. The morphism $\nu$ is an $h$-covering, and $X^\nu$ lies in $\Schfft_{S}$ by \cref{prop:normalization-has-enough-qs}. We set $W = (X^{\nu} \times_X X^{\nu})_{\redu}$ for the reduced subscheme of the fibre product,  let 
\[ 
	q_1,q_2 \colon W \to X^{\nu} 
\] 
be the two projections, and denote by $p_1,p_2 \colon W^{\awn} \to X^{\nu}$ the compositions of $q_1,q_2$ with the absolute normalization map of $W$. There is then the commutative diagram
\[ 
	\begin{tikzcd}
	\Hom_S(X^{\awn},Y) \arrow[d,hook] \arrow[r,"\circ \omega"] 
	& \Hom_S(X^{\nu},Y) \arrow[d,hook] \arrow[r,shift left=.25em, "\circ p_1"] \arrow[r,shift right = .25em, "\circ p_2" '] 
	& \Hom_S(W^{\awn},Y) \arrow[d] 
	\\
	\Hom^{\che}_S(X^{\awn}_{\et},Y_{\et}) \arrow[r,"\circ \omega_\et"] 
	& \Hom^{\che}_S(X^{\nu}_{\et},Y_{\et}) \arrow[r,shift left=.25em, "\circ p_{1,\et}"] \arrow[r,shift right=.25em,"\circ p_{2,\et}"']
 	& \Hom^{\che}_S(W^{\awn}_{\et},Y_{\et}).
\end{tikzcd}
\] 
Note that the rows are equalizers: this follows from \cite[Theorem 5.1, Corollary~5.9]{CHWvoevodskyanabelian}, together with the fact that $X^\nu$ is absolutely weakly normal (see \cref{lem:normalabsnorm}),  that absolute weak normalization is insensitive to infinitesimal thickenings, and that universal homeomorphisms induce equivalences on \'etale topoi. Suppose that 
\[ 
	\varphi \colon X^{\awn}_{\et} \to Y_{\et} 
\]  
is a chivalrous morphism of \'etale topoi over $S$. If the theorem holds for normal schemes, the composition $\varphi \circ \omega_{\et}$, which is chivalrous, comes from an $S$-morphism $g \colon X^{\nu} \to Y$. We claim that $g \circ p_1 = g \circ p_2$. Indeed, since the lower row is commutative, we have the equality 
\[ 
	(g \circ p_1)_{\et} = (g \circ p_2)_{\et} 
\] 
of morphisms of \'etale topoi, and since $W^{\awn} \to W$ is a universal homeomorphism, hence induces an equivalence on \'etale topoi, we find $(g \circ q_1)_{\et} = (g \circ q_2)_{\et}$. Since $W$ is reduced and of finite type over~$S$, \cref{injectivitystrengthening} shows that $g \circ q_1 = g \circ q_2$. It then follows that $g \circ p_1 = g \circ p_2$, by composing with the map $W^{\awn} \to W$. Thus, the upper equalizer diagram shows that $g$ descends to an $S$-map $f \colon X^{\awn} \to Y$. The commutativity of the whole diagram, together with the fact that $\circ \omega_\et$ is a monomorphism (because the lower row is an equalizer), shows that $f_{\et} = \varphi$.  
\end{proof}

\subsection{Spreading out maps and detecting maps between \'etale topoi} 
\label{sec:proof-of-relative-thm}

\begin{proposition} \label{prop:extend-from-generic-fibre}
Let $S$ be a noetherian Nagata scheme that is geometrically unibranch and such that the function field of every irreducible component is a sub-$p$-adic field. Let $X,Y$ be schemes of finite type over $S$ and assume $X$ is normal and that $X \in \Schfft_S$. 
Let $\ph \colon X_\et \to Y_\et$ be a fibrewise pinned morphism over $S_\et$. Then there is a unique $S$-morphism $f \colon X \to Y$ such that for all normal quasi-sections $q \colon W \to X$ we have equality
\[
	f_\et \circ q_\et = \ph \circ q_\et.
\]
\end{proposition}

\begin{proof}
Denote by $S_0$ the disjoint union of all generic points of $S$, and the base change with $S_0 \to S$ by~$(-)_0$. 
\smallskip

\textit{Step 1:} 
We 
first show the uniqueness of the claimed $S$-morphism $f \colon X \to Y$. Indeed, since $X_0 \subseteq X$ is schematically dense, the map $f$ is determined by the restriction $f_0 \colon X_0 \to Y_0$ and in fact by the restriction to all closed points $z \in X_0$. Let $Z_z \subseteq X$ be the Zariski closure of $z$ in $X$. Then a suitable open $W_z \subseteq Z_z$ determines a normal quasi-section $q_z \colon W_z \to X$ over $S$ which contains $z$ in its image. Then $f_0|_z = (f \circ q_z)|_z$ is determined uniquely by 
\[
(f \circ q_z)|_{z,\et} = (f_\et \circ q_{z,\et})|_{z_\et} = (\ph \circ q_{z,\et})|_{z_\et}.
\]
Hence $\ph$ determines $f$ uniquely if it exists. 
\smallskip

\textit{Step 2:} 
Now we assume the proposition is proven with some $f \colon X \to Y$ for some $\ph \colon X_\et \to Y_\et$ and claim that for any Zariski open $V \subseteq Y$ the claim also follows for the restriction $\ph_V \colon \ph^\ast V_\et \to V_\et$, more precisely, that $f$ restricts to a map $f_V \colon U = \ph^\ast V \to V$  that satisfies the claim of the proposition. 
Indeed, for every normal quasi-section $q \colon W \to U$ we compute the preimage $(f \circ q)^{-1}(V)$ as
\[
(f \circ q)_\et^\ast(V) = (\ph \circ q_\et)^\ast(V) = q_\et^\ast(U) = W,
\]
so that $f \circ q$ factors over the inclusion $j_V \colon V \to Y$. Since $U \to S$ has enough normal quasi-sections, this shows that $f$ maps all closed points of its fibres over $S$ into $V$. Since $V$ is open, 
$f$ restricts to an $S$-morphism $f_V \colon U  \to V$. Denote by $j_U \colon U \to X$ the inclusion. For any normal quasi-section $q \colon W \to  U$ over $S$ we compute
\[
j_{V,\et} \circ (f_V \circ q)_\et = (j_V \circ f_V \circ q)_\et = (f \circ (j_U \circ q))_\et = \ph \circ (j_U \circ q)_\et = j_{V,\et} \circ \ph_V \circ q_\et.
\]
This shows the claim of step 2, because we may cancel $j_{V,\et}$ by \cref{lem:factorisation}.

\smallskip

\textit{Step 3:}  In the next step we reduce to when $Y$ and $S$ are affine.  Let $\{S_i\}_{i \in I}$ be an affine open covering of $S$, and for each $i$, let $\{V_{ij}\}_{j \in J_i}$ be an affine open covering of $Y_i = Y \times_S S_i$. This induces a Zariski open covering of $X$ by $U_{ij} = \ph^* (V_{ij}) \subseteq X$. Note that the induced morphisms 
\[ 
	\ph_{ij} \colon U_{ij,\et} \to V_{ij,\et} 
\] 
are fibrewise pinned by \cref{rmk:restrictionfiberwisepinned}, and that the $U_{ij}$ are normal and universally open over $S$. If the statement holds in the affine case (target and base affine), then we obtain maps $f_{ij} \colon U_{ij} \to V_{ij}$.  By step 1 and 2 these maps $f_{ij}$ agree on overlaps and thus glue to an $S$-map $f \colon X \to Y$.  

We still need to show that this map $f$ satisfies the claim of the proposition with respect to $\ph$ and all normal quasi-sections $q \colon W \to X$ over $S$.  But the claimed equality $f_\et \circ q_\et = \ph \circ q_\et$ is \'etale local on $Y_\et$ and thus can be verified by passing to an \'etale covering, where it holds by construction of $f$.

Hence we may, and do, assume that $Y$ and $S$ are affine.

\smallskip

\textit{Step 4:}  By \cite[Lemma~3.6]{CHWvoevodskyanabelian} base change of $\ph$ induces a pinned $S_0$-morphism
\[
	\ph_0 \colon X_{0,\et} \to Y_{0,\et}.
\]
By \cref{intro_thm:etale_reconstructionB}, applied to connected components, there exists a unique $S_0$-morphism $f_0 \colon X_0 \to Y_0$ with $f_{0,\et} = \ph_0$.
Suppose now that $q \colon W \to X$ is a normal quasi-section over $S$. Since the composite $\varphi \circ q_\et$ is fibrewise pinned, 
\cref{cor:quasifinitesection-general} shows that it arises from a unique map of $S$-schemes $f_q \colon W \to Y$.  Passing to the generic fibres, we obtain equalities
\[
(f_0 \circ q_0)_\et = f_{0,\et} \circ q_{0,\et} = \ph_0 \circ q_{0,\et} = (\ph \circ q_\et)_0 = (f_{q,\et})_0 = (f_{q,0})_\et.
\]
\cref{injectivitystrengthening} shows that $f_0 \circ q_0$ and $f_{q,0}$ agree. If $f_0$ extends to an $S$-map $f: X \to Y$, then, since $W_0$ is dense in $W$, this implies $f \circ q = f_q$. Consequently, then 
\[
	(f \circ q)_\et = f_{q,\et} = \varphi \circ q_\et.
\]
It thus suffices to show that $f_0$ extends to an $f$. 

\smallskip

\textit{Step 5:} 
Let $\Gamma \subseteq X \times_S Y$ be the Zariski closure of the graph of $f_0 \colon X_0 \to Y_0$, and let $\sigma \colon \Gamma \to X$ denote the projection. For every normal quasi-section $q \colon W \to X$ the map $f_q$ that extends the map $f_{q,0} = f_0|_{W_0}$ shows that $q(W)$ is contained in the image of $\sigma$. Since $X \to S$ has enough quasi-sections by \cref{prop:normalization-has-enough-qs}, the image of $\sigma$ contains all closed points of every fibre of $X \to S$. Since $X \to S$ is of finite type and $\sigma(\Gamma)$ is constructible, the map $\sigma$ is surjective.  

The locus in $\Gamma$ where $\sigma$ is quasi-finite is open 
and contains the fibres above all points of $X$ of codimension $1$ since $X$ is normal and $\sigma$ is birational and surjective. Then \cite[Corollaire~4.4.9]{GrothendieckEGAIII1} implies that there exists an open $j \colon U \subseteq X$ such that $\sigma$ restricts to an isomorphism $\sigma^{-1}(U) \to U$ such that $X-U$ has codimension at least $2$. Since $X$ is normal, $j_\ast \dO_U = \dO_X$, and since $Y$ and $S$ are affine, it follows that $f_U$ extends uniquely to an $S$-morphism $f \colon X \to Y$.
\end{proof}

\begin{proposition} \label{prop:detect-effect-on-topoi}
Let $S$ be a noetherian Nagata scheme. Let $X,Y$ be schemes of finite type over $S$ and assume that $X$ is reduced, that $Y \to S$ is separated and that $X \to S$ is universally open. Let $\ph \colon X_\et \to Y_\et$ be a morphism over $S_\et$.  Assume that for every \'etale open $V \in \Et_Y$ with induced morphism $\ph_V \colon  (\ph^\ast V)_{\et} \to V_{\et},$ there is an $S$-morphism $f_V \colon \ph^\ast V \to V$ such that, for all normal quasi-sections $q \colon W \to \ph^\ast V$, the morphisms
\[
	f_{V,\et} \circ q_\et  = \ph_V \circ q_\et
\]
agree. Then $f = f_Y$ induces a geometric morphism $f_\et \colon X_\et \to Y_\et$ over $S$ which is equivalent to $\ph$. 
\end{proposition}

\begin{proof}
This is proved analogously to \cite[Proposition~5.11]{CHWvoevodskyanabelian}. By  
\cref{rec:sitesandtopoi}, the map $\ph$ induces a morphism of sites $\ph^* \colon \Et_X \to \Et_Y$ over $\Et_S$, and it is enough to show that there is a natural isomorphism 
\[ 
	\gamma \colon \ph^* \isomto  f^\ast
\] 
of morphisms of sites, compatible with the maps to $\Et_{S}$.  

Let $j_V \colon V \to Y$  be an \'etale open of $Y$, and set 
$j_U \colon U = \ph^\ast(V) \to X$ for the pullback along $\ph$.  We first prove that the diagram 
\begin{equation}
\label{eq:pullback-square-commutes}
	\begin{tikzcd}
	\ph^\ast(V) \arrow[r,"f_V"]  \arrow[d,"j_U"] 
	& V \arrow[d,"j_V"] 
	\\
	X \arrow[r,"f"] 
	& Y	
	\end{tikzcd}
\end{equation}
commutes. Let $q \colon W \to X$ be a normal quasi-section and let $q' \colon W \times_X \ph^\ast(V) \to \ph^\ast(V)$ be the pullback of $q$ along $j_U$.  The assumed equality $\ph \circ q_\et = f_\et \circ q_\et$ contains implicitly an isomorphism
\[
	\alpha_V \colon W \times_X \ph^\ast(V) = (\ph \circ q_\et)^*V \isomto (f_\et \circ q_\et)^*V =  W \times_Y V
\] 
over $W$. 
We claim that $f_V \circ q' = \pr_V \circ \alpha_V$, where $\pr_V \colon W \times_Y V \to V$  
is the projection. We first note that since $q^* \ph^* \cong q^* f^*$ is a natural isomorphism, we obtain an isomorphism
\[ 
	\ph_V \circ q'_{\et} \cong  (\pr_V \circ \alpha_V)_{\et}. 
\] 
Since $q'$ is a normal quasi-section, we see that $\ph_V \circ q'_{\et} = (f_V \circ q')_{\et}$. This implies that $(f_V \circ q')_{\et}  \cong (\pr_V \circ \alpha_V)_{\et}$, and \cref{injectivitystrengthening} shows that $f_V \circ q' = \pr_V \circ \alpha_V$. 

Therefore there is a diagram with a commutative outer rectangle and a cartesian left square 
\[
	\begin{tikzcd}
	W \times_X \ph^\ast(V) \arrow[r,"{q'}"]  \arrow[d, "j_W"]
	&  \ph^\ast(V) \arrow[r,"f_V"]  \arrow[d,"j_U"] 
	& V \arrow[d,"j_V"] 
	\\
	W \arrow[r,"q"] 
	& X \arrow[r,"f"] 
	& Y	
	\end{tikzcd}
\]
It follows that $j_V \circ f_V$ agrees with $f \circ j_U$ on the image of $q'$. Since $X \to S$ has enough normal quasi-sections by \cref{prop:enough-normal-quasi-sections}, the locus where $j_V \circ f_V$ agrees with $f \circ j_U$  contains all closed points of all fibres and thus is all of $\ph^*(V)$. Hence the diagram \labelcref{eq:pullback-square-commutes} commutes as claimed. 

The commutativity of  \labelcref{eq:pullback-square-commutes}  induces a map in $\Et_X$
\[
	\gamma_V \colon \ph^*(V) \longrightarrow X \times_Y V.
\]
Thus $\gamma_V$ is \'etale, and a map between \'etale opens of $X$. Base changing to a normal quasi-section $q \colon W \to X$ we find as above that the base change of $\gamma_V$ 
\[
\gamma_V \times \id_W \colon \ph^*(V) \times_X W \longrightarrow W \times_Y V
\]
is an isomorphism of \'etale opens of $W$. It follows that the $X$-map $\gamma_V$ is an isomorphism restricted to every point $x \in X$ that lies in the image of such a normal quasi-section. The set of such points contains all closed points of all fibres of $X \times_Y V \to S$. The locus in $X$ where $\gamma_V$ is not an isomorphism is closed, hence contains a closed point of a fibre over $S$. This shows that this locus is empty and hence $\gamma_V$ is an isomorphism. 

We now show that the family of $(\gamma_V)_{V \in \Et_Y}$ is a natural transformation. Let $j_V \colon V_1 \to V_2$ be a morphism in $\Et_Y$, and set 
\[
	j_U = \ph^\ast(j_V) \colon U_1 = \ph^\ast(V_1) \to U_2 = \ph^\ast (V_2)
\]
for the base change. The commutativity of the diagram 
\[
	\begin{tikzcd}
	\ph^\ast(V_1) \arrow[r,"\gamma_{V_1}"]  \arrow[d,"j_U"] 
	& X \times_Y V_1  \arrow[d,"j_V"] 
	\\
	\ph^\ast(V_2) \arrow[r,"\gamma_{V_2}"] 
	& X \times_Y V_2	
	\end{tikzcd}
\]
follows again by restricting to enough quasi-sections. One similarly shows that the natural isomorphism $\gamma$ is compatible with the maps to $S$.
\end{proof}

\begin{proof}[Proof of \cref{thm:etalereconstructiongeneral}]
By \cref{lem:h-descent} we may assume that $X$ is normal, in particular $X = X^\awn$. Since \cref{injectivitystrengthening} proves injectivity of
\[
	(-)_\et \colon \Hom_S(X,Y) \la \Hom^{\che}_{S}(X_{\et},Y_{\et})
\]
we must show that a chivalrous $\ph: X_\et \to Y_\et$ over $S_\et$ comes from an $S$-morphism of schemes.  We explained in \cref{rmk:basechangechivalrous} that chivalrous implies fibrewise pinned, so that \cref{prop:extend-from-generic-fibre} gives us an $S$-morphism $f \colon X \to Y$ that agrees with $\ph$ upon restriction to normal quasi-sections. Moreover, for any \'etale open $V \to Y$ the same applies to the restriction $\ph_V \colon \ph^*(V)_\et \to V_\et$, which is also fibrewise pinned by \cref{rmk:restrictionfiberwisepinned}. 
Then \cref{prop:detect-effect-on-topoi} shows that $f_\et = \ph$. 
\end{proof}


\end{document}